\documentclass[10pt]{amsart}

\usepackage{bbding, amsmath,  amsfonts, amssymb,mathtools }
\usepackage{blindtext} 
\usepackage{bm}
\usepackage[colorlinks=true]{hyperref}
\usepackage{cleveref}
\usepackage{textcomp}
\usepackage{bold-extra}
\usepackage{float}
\usepackage{cancel}
\usepackage{graphicx,array}
\newcolumntype{L}[1]{>{\raggedright\let\newline\\\arraybackslash\hspace{0pt}}m{#1}}
\newcolumntype{C}[1]{>{\centering\let\newline\\\arraybackslash\hspace{0pt}}m{#1}}
\newcolumntype{R}[1]{>{\raggedleft\let\newline\\\arraybackslash\hspace{0pt}}m{#1}}
\usepackage{xcolor}
\usepackage[all]{xy}
\usepackage{tikz}
\usepackage{tikz-cd}
\usetikzlibrary{cd}
\usepackage{pgfplots}
\pgfplotsset{width=7cm,compat=1.8}
\usetikzlibrary{arrows,intersections,shapes,snakes,positioning}
\tikzset{
    >=stealth',
    punkt/.style={
           rectangle,
           rounded corners,
           draw=black, very thick,
           text width=8em,
           minimum height=2em,
           text centered},
    pil/.style={
           ->,
           thick}
}
\usepackage{units}
\usepackage{graphicx}
\usepackage{mathrsfs,stmaryrd}
\definecolor{myred}{rgb}{0.5, 0, 0}
\definecolor{linkred}{rgb}{0.8, 0, 0}
\definecolor{mygreen}{rgb}{0, 0.5, 0}
\definecolor{citegreen}{rgb}{0, 0.8, 0}
\definecolor{yqqqqq}{rgb}{0.5019607843137255,0,0}
\definecolor{bluetable}{rgb}{0,0.2,0.5}
\usepackage[T1]{fontenc} 
\usepackage{microtype} 
\usepackage[sc]{mathpazo}
\usepackage{eulervm}
\usepackage[english]{babel} 

\usepackage[hang, small,labelfont=bf,up,textfont=it,up]{caption} 
\usepackage{booktabs} 
\usepackage{lettrine} 

\usepackage{enumitem} 
\setlist[itemize]{noitemsep} 

\newtheorem{Thm}{Theorem}[subsection]
\newtheorem*{Thm*}{Theorem}
\newtheorem{Lem}{Lemma}[subsection]
\newtheorem{Prop}{Proposition}[subsection]
\newtheorem{Cor}{Corollary}[subsection]
\newtheorem*{Cor*}{Corollary}
\newtheorem{Conj}{Conjecture}[subsection]
\theoremstyle{definition}
\newtheorem{Defn}{Definition}[subsection]
\newtheorem*{Remark*}{Remark}
\newtheorem{Remark}{Remark}[subsection]
\newtheorem*{Notation}{Notation}
\newtheorem{Ex}{Example}

\numberwithin{equation}{section}

\usepackage[colorlinks=true]{hyperref}
\usepackage{cleveref}

\hypersetup{
    citecolor=citegreen
    ,linkcolor=linkred
    ,urlcolor=bluetable
    }

\AtBeginDocument{%
   \def\MR#1{}
}

\def\ldb{\mathopen{\{\!\{}} \def\rdb{\mathclose{\}\!\}}}

\long\def\comment#1{}

\makeatletter
\newcommand{\oset}[3][0ex]{%
  \mathrel{\mathop{#3}\limits^{
    \vbox to#1{\kern-2\ex@
    \hbox{$\scriptstyle#2$}\vss}}}}
\makeatother
\newcommand*\circled[1]{\tikz[baseline=(char.base)]{
            \node[shape=circle,draw,inner sep=2pt] (char) {#1};}}
\newcommand{\acfib}{\oset{\sim}{\twoheadrightarrow}}
\newcommand{\we}{\oset{\sim}{\to}}
\newcommand{\fibr}{\twoheadrightarrow}

\newcommand{\cofi}{{\hookrightarrow}}

\newcommand\adjunct[4]{\xymatrix@C=2pc@R=2pc{#1\ar @<1.25ex>[rr]^{#3}&\perp&#2\ar @<1.25ex>[ll]^{#4}}}
\def\Z{\mathbb{Z}}
\def\N{\mathbb{N}}

\def\D{\mathbb{D}}

\def\R{\mathbb{R}}

\def\L{{\boldsymbol{L}}}
\def\R{{\boldsymbol{R}}}

\newcommand{\lieg}{\mathfrak{g}}
\newcommand{\gl}{\mathfrak{gl}}

\newcommand{\nn}{{\underline{\mathbf{n}}}}

\def\O{\mathcal{O}}

\def\C{\mathcal{K}}
\def\B{\mathcal{B}}
\def\AA{\mathtt{A}}
\def\BB{\mathtt{B}}
\def\CC{\mathtt{C}}
\def\DD{\mathtt{D}}

\def\Act{\mathtt{Act}}
\def\NAct{\mathtt{NAct}}
\def\CoAct{\mathtt{CoAct}}
\def\CoGrp{\mathtt{CoGrp}}
\def\Grp{\mathtt{Grp}}
\def\NGrp{\mathtt{NGrp}}
\def\NAff{\mathtt{NAff}}
\def\Vect{\mathtt{Vect}}

\def\DGVect{\mathtt{DGVect}}

\def\Alg{\mathtt{Alg}}
\def\PPA{\mathtt{PPAlg}}
\def\CPA{\mathtt{CPAlg}}

\def\Aff{\mathtt{Aff}}
\def\PAff{\mathtt{PAff}}
\def\PGAff{\mathtt{PGAff}}
\def\NPAff{\mathtt{NPAff}}
\def\NPGAff{\mathtt{NPGAff}}

\def\DGAff{\mathtt{DGAff}}
\def\Mod{\mathtt{Mod}}

\newcommand{\Bimod}{\mathtt{Bimod}}
\def\cAlg{\mathtt{CommAlg}}
\def\cHopf{\mathtt{CHopf}}

\def\DGA{\mathtt{DGA}}
\def\DGLA{\mathtt{DGLA}}
\def\DGPA{\mathtt{DGPA}}

\def\CDGPA{\mathtt{CDGPA}}
\def\DGPPA{\mathtt{DGPPA}}
\def\CDGA{\mathtt{CDGA}}

\def\Cat{{\mathtt{Cat}}}
\def\Ho{{\mathtt{Ho}}}
\def\op{\mathtt{op}}

\newcommand{\Sh}{{\mathrm{Sh}}}
\newcommand{\CE}{{\mathrm{CE}}}
\newcommand{\old}{{\mathrm{old}}}
\newcommand{\dual}{\text{\textasciicaron}}

\newcommand{\pt}{{\mathrm pt}}

\newcommand{\acts}{{\, \curvearrowright\,} }

\newcommand{\Hom}{{\mathrm{Hom}}}

\newcommand{\GL}{{\mathrm{GL}}}

\newcommand{\Lie}{{\mathrm{Lie}}}
\newcommand{\End}{\mathrm{End}}
\newcommand{\Ad}{\mathrm{Ad}}
\newcommand{\Aut}{\mathrm{Aut}}
\newcommand{\Span}{\mathrm{Span}}

\newcommand{\BRST}{{\mathrm{BRST}}}
\newcommand{\NGa}{{\mathrm{N}\mathbb{G}_a}}
\newcommand{\NGm}{{\mathrm{N}\mathbb{G}_m}}

\newcommand{\DDer}{{\mathbb{D}\mathrm{er}}}
\newcommand{\Der}{\mathrm{Der}}

\newcommand{\Rep}{{\mathrm{Rep}}}
\newcommand{\HH}{{\mathrm{H}}}

\newcommand{\ev}{\mathrm{ev}}
\newcommand{\n}{{\natural}}

\newcommand{\Spec}{{\rm{Spec}}}
\newcommand{\Sp}{{\rm{Sp}}}

\newcommand{\Sym}{{\rm{Sym}}}
\newcommand{\Symg}{{\rm{Sym(\mathfrak{g})}}}

\newcommand{\id}{{\rm{Id}}}

\newcommand{\tr}{{\rm{tr}}}

\newcommand{\im}{{\rm{im}}}

\newcommand{\odd}{{\mathrm{odd}}}
\begin{document}


\title[Noncommutative derived Poisson reduction]{Noncommutative derived Poisson reduction} 
\author{Stefano D'Alesio} 
\address{Stefano D'Alesio, Departement Mathematik,
ETH Z\"urich,
8092 Z\"urich, Switzerland}
\email{stefano.dalesio@math.ethz.ch}
\date{} 


\begin{abstract}
\noindent
In this paper we propose a procedure for a noncommutative derived Poisson reduction, in the spirit of the Kontsevich-Rosenberg principle: ``a noncommutative structure of some kind on $A$ should give an analogous commutative structure on all schemes $\Rep_n(A)$''. We use double Poisson structures as noncommutative Poisson structures and noncommutative Hamiltonian spaces --- as first introduced by M. Van den Bergh --- to define (derived) zero loci of Hamiltonian actions and a noncommutative Chevalley-Eilenberg and BRST constructions, showing how we recover the corresponding commutative constructions using the representation functor. In a dedicated final short section we highlight how the categorical properties of the representation functor lead to the natural introduction of new interesting notions, such as noncommutative group schemes, group actions, or Poisson-group schemes, which could help to understand the previous results in a different light, and in future research generalise them into a broader, clearer correspondence between noncommutative and commutative equivariant geometry.
\end{abstract} 

\setcounter{tocdepth}{1}     
\maketitle
{
{
\color{bluetable}
\hypersetup{linkcolor=bluetable}
\tableofcontents
}
}


\section{Introduction}
\label{1}

A known principle in noncommutative geometry (\cite{KR}) says that every geometrically meaningful structure on an (associative, unital) algebra $A\in \Alg_k$ should induce the corresponding geometric structure on the scheme of representations $\Rep_n(A)$ in a $n$-dimensional vector space. Noncommutative Poisson geometry was worked out first by W. Crawley-Boevey who defined a Poisson structure on the character scheme $\Rep_n(A)\sslash \GL_n$ through his definition of $\HH_0$-Poisson structures (\cite{Cr}) and then by M. Van den Bergh in \cite{VdB}, who made the observation that a Poisson bracket on the full representation scheme $\Rep_n(A)$ shall be defined on the generators $\{a_{ij},b_{kl}\}$, and because it depends on four indices, it is natural to assume that it comes from a double bracket $\ldb-,-\rdb : A \otimes A \to A \otimes A $, with some properties that ensure that the induced bracket on the representation scheme is indeed a Poisson bracket. 

M. Van den Bergh also defined, using double Poisson structures, a noncommutative version of Hamiltonian spaces, essentially double Poisson algebras with a distinguished `gauge' element $\delta \in A$ that acts via the double Poisson bracket as the universal derivation on the algebra~\eqref{eq:Hamiltoniandelta}. This ensures that the corresponding action of the gauge group $\GL_n \acts \Rep_n(A)$ on the representation scheme is a Hamiltonian action, with moment map $\mu_n:\Rep_n(A) \to \gl_n^{(\ast)}$ described as the evaluation of a representation on the element $\delta$. One can define a noncommutative version of Poisson reduction then, by considering the quotient algebra $A\slash \langle \delta\rangle$ by the two sided ideal generated by $\delta$, which is a noncommutative counterpart of the zero locus: $\Rep_n(A\slash \langle \delta \rangle) = \mu_n^{-1}(0)$, so that its $\GL_n$-quotient is the Poisson reduction $\mu_n^{-1}(0)\sslash \GL_n$. These ideas appear in various forms in the first papers on noncommutative symplectic and Poisson geometry, such as \cite{K,G1,BLB,G2,CEG,EG,VdB,VdB1}.

In this paper, we elaborate an idea from V. Ginzburg (who first defined some `noncommutative BRST complexes' in \cite{G3}) and we work out in details a possible procedure to do noncommutative Poisson reduction in a derived fashion: we add variables in positive homological degrees to kill relations instead of considering quotients, and we add other variables in negative homological degrees as some sort of Chevalley-Eilenberg generators instead of considering invariants. In order to clarify our definitions involved in this `derived Poisson reduction' in the noncommutative world, it is conveninet to first recall briefly the commutative construction of derived Poisson reduction, in the style of \cite{Ca,Sa}. 

We start from a Poisson algebra $B$, a Hamiltonian group scheme action of a reductive group $G \acts X=\Spec(B)$, with (co)moment Poisson map: $\Sym(\lieg) \to B$ (the Poisson structure on $\Sym(\lieg)$ is the natural extension of the Lie bracket). We first define the derived zero locus of the corresponding map of schemes $\mu: X \to \lieg^\ast$ as the homotopy pull-back in the category of dg schemes (dually, the homotopy push-out diagram in the category of commutative dg algebras over):
\begin{equation}
\label{eq:1}
B \otimes_{\Sym(\lieg)}^\L k \quad \left( \leftrightarrow \quad X \times_{\lieg^\ast}^h \pt \right)\, .
\end{equation}
We then apply the Chevalley-Eilenberg functor $C(\lieg, - )=\Hom_k(\Sym(\lieg[1]),-)$ to the derived zero locus and obtain the classical BRST complex
\begin{equation}
\label{eq:2}
C(\lieg, B\otimes_{\Sym(\lieg)}^\L k) \quad \left( \simeq C(\lieg,B)\otimes^\L_{C(\lieg,\Sym(\lieg))} C(\lieg,k) \quad \leftrightarrow \quad [X/G] \times^h_{[\lieg^\ast/G]} [\pt/G] \right)\, ,
\end{equation}
as a derived model for the algebra of functions on the reduced space $\mu^{-1}(0)\sslash G$.


\subsection{Summary of results}
\label{1.1}

In the noncommutative context we consider a dg algebra $A \in \DGA_S$ over $S$, a finite dimensional algebra of orthogonal idempotents $S=kI$ (path algebra of a quiver with vertex set $I$ and no arrows). When we consider representations over $S$ in a vector space $V$ we need to specify a dimension vector $\nn\in \N^I$ (of total dimension = $\dim V$) or in other words fix the representation $\rho_\nn :S\to \End(V)$ sending the $i$-th orthogonal idempotent to the corresponding one according to the decomposition given by the dimension vector $\nn$. The representation functor has the form:
\begin{equation}
\label{eq:3}
\begin{aligned}
(-)_\nn : &\,\,\DGA_S \to \CDGA_k\, \\
& A \longmapsto A_\nn = \O \Rep_\nn(A)\, ,
\end{aligned}
\end{equation}
where $\Rep_\nn(A)$ denotes the scheme of representations of $A$ that agree with $\rho_\nn$ through the structure map $S \to A$. When $A$ is equipped with a double Poisson algebra structure, there is an induced Poisson structure on $A_\nn$ defined on its generators by:
\begin{equation}
\label{eq:4}
\{ a_{ij} ,b_{kl} \} = \ldb a,b\rdb'_{kj} \ldb a,b\rdb''_{il}\, .
\end{equation}
In other words the representation functor enriches to a functor between the categories of (dg) double Poisson algebras, and commutative (dg) Poisson algebras:
\begin{equation}
\label{eq:5}
(-)_\nn: \DGPPA_S \to \CDGPA_k\, .
\end{equation}
Let now $A \in \PPA_S$ be a (ungraded) double Poisson algebra. The representation scheme $\Rep_\nn(A)$ comes with a natural action $G_S \acts \Rep_\nn(A)$ of the gauge group of $S$-preserving automorphisms $G_S \subset \GL(V)$, which in this case is just a product of general linear groups $\GL_\nn:=\prod_i \GL_{n_i}$. The corresponding Lie algebra $\lieg=\gl_\nn$ can be obtained as the representation scheme of the path algebra of a quiver with vertex set $I$ and one loop $t_i$ on each vertex, $T_S(L)$ ($L$ is the linear span of the loops). As it turns out there is a natural double Poisson structure on $T_S(L)$ such that the induced Poisson structure on
\begin{equation}
\label{eq:6}
(T_S(L))_\nn = \O (\gl_\nn)= \Sym(\gl_\nn^\ast) \cong \Sym(\gl_\nn)
\end{equation}
is the standard extension of the Lie algebra structure on $\gl_\nn$ (where in the last identification we used the canonical isomorphism induced by the trace). Once we know this, we see that the action $\GL_\nn \acts \Rep_\nn(A)$ is Hamiltonian exactly when $A$ comes with a morphism of double Poisson algebras $T_S(L)\to A$, with the additional property that the image of a loop $t_i \mapsto \delta_i $ has Poisson bracket with any $a \in A$:
\begin{equation}
\label{eq:6*}
\ldb \delta_i ,a\rdb = ae_i\otimes e_i - e_i \otimes e_i a \,,
\end{equation}
(the $i$-th component of the universal derivation). In other words it is possible to define of a category of noncommutative Hamiltonian spaces as a full subcategory (objects with structure map having the property~\eqref{eq:6*}) of the under category $\PPA_{T_S(L)}^H\subset \DGPPA_{T_S(L)}$ in such a way that:
\begin{Thm*}[\S\ref{4.2}]
The representation functor enriches to a functor between the category of noncommutative Hamiltonian spaces and commutative Hamiltonian $\gl_\nn$-spaces:
\begin{equation}
\label{eq:7}
\begin{tikzcd}
  \PPA_{T_S(L)}^H  \arrow{r}{(-)_\nn} \arrow{d}{}
    & \CPA_{\Sym(\gl_\nn)}^H \arrow{d}{}  \\
\DGA_{T_S(L)} \arrow{r}{(-)_\nn}
&\CDGA_{\Sym(\gl_\nn)}
 \end{tikzcd}
 \end{equation}
where the vertical functors forget the Poisson structures and view the algebras as dg algebras placed in degree zero. Moreover the representation functor at the level $(-)_\nn: \DGA_{T_S(L)} \to \CDGA_{\Sym(\gl_\nn)}$ is cocontinuous (preserves small colimits), so in particular it preserves coproducts.
\end{Thm*}
Once this is clear it is natural to give the necessary definitions and a procedure to do noncommutative derived Poisson reduction, simply by substituting the constructions in~\eqref{eq:1} and~\eqref{eq:2} by the corresponding noncommutative ones. 

The noncommutative analogue of the zero locus of the Poisson moment map is the coproduct over $T_S(L)$ of the Hamiltonian algebra $T_S(L) \to A$ and $S$ (viewed as a $T_S(L)$-algebra via the standard projection that sends $L$ to zero):
\begin{equation}
\label{eq:8}
A/\langle L \rangle = A \amalg_{T_S(L)} S  \xmapsto{(-)_\nn}A_\nn \otimes_{\Sym(\gl_\nn)} k \, .
\end{equation}
We can therefore define a noncommutative derived zero locus substituting the coproduct with the derived coproduct, in such a way that we recover the classical derived zero locus:
\begin{Thm*}[\S\ref{4.2}]
The following model of noncommutative derived zero locus corresponds, under the representation functor, to the classical derived zero locus - the Koszul complex: 
\begin{equation}
\label{eq:8}
A \amalg_{T_S(L)}^\L S \cong A \amalg_S T_S(L[1])  \xmapsto{(-)_\nn} A_\nn \otimes^\L_{\Sym(\gl_\nn)} k \cong A_\nn \otimes_k \Sym(\gl_\nn[1])\, .
\end{equation}
\end{Thm*} 
We denote this specific model for the noncommutative derived zero locus by $\Sh(A) = A\amalg_S T_S(L[1])$, because it is some sort of a generalised `Shafarevich complex'. The next step for doing Poisson reduction is a noncommutative Chevalley-Eilenberg construction, which we define to be the following coproduct (with twisted differential, as in the commutative case --- details in \S\ref{3.4}):
\begin{equation}
\label{eq:9}
\begin{aligned}
\CE   : &\,\, \DGPPA_{T_S(L)} \to \DGA_{T_S(L \oplus L^\ast[-1])}\\
& A \xmapsto{\quad \quad} A \amalg_{T_S(L)} T_S(L\oplus L^\ast[-1])\, .
\end{aligned}
\end{equation}
The reader who is wondering why we momentaneously forget the Poisson structure is encouraged to read the details of this construction in \S\ref{3.4}, especially Remark~\ref{rem:CE}. When we start from a noncommutative Hamiltonian space $A \in \PPA_{T_S(L)}^H$ and apply the Chevalley-Eilenberg construction to the Shafarevich complex we obtain a noncommutative version of the BRST complex:
\begin{equation}
\label{eq:10}
\BRST(A) := \CE(\Sh(A)) \cong A \amalg_S T_S( L[1] \oplus L^\ast[-1])\, ,
\end{equation}
now equipped with the natural double Poisson structure which comes from $A$ and the natural pairing between $L, L^\ast$. We obtain the following result:
\begin{Thm*}[\S\ref{4.3}]
There is a commutative diagram between the noncommutative and the commutative BRST construction:
\begin{equation}
\label{eq:11}
\begin{tikzcd}
  \PPA_{T_S(L)}^H  \arrow{r}{(-)_\nn} \arrow{d}{\rm{BRST}}
    & \CPA_{\Sym(\gl_\nn)}^H \arrow{d}{\rm{BRST}}  \\
\DGPPA_{T_S(L \oplus L[1] \oplus L^\ast[-1])} \arrow{r}{(-)_\nn} 
&\CDGPA_{\Sym(\gl_\nn \oplus \gl_\nn[1] \oplus \gl_\nn^\ast[-1])}
 \end{tikzcd}
\end{equation}
\end{Thm*}

We conclude this introduction by providing a `dictionary', a summary of the above-mentioned noncommutative constructions and their commutative counterparts.

\begin{figure}[H]
\caption{Dictionary between noncommutative and commutative geometry ($\lieg=\gl_\nn$ in the table).}
\label{fig:dic}
\begin{center}
    \begin{tabular}{| C{2.5cm} | | C{4.7cm}  | C{4.7cm} |}
    \hline
    \emph{\underline{`Dictionary'}} & Noncommutative geometry$^{(\op)}$ & Commutative geometry$^{(\op)}$ \\ \hline
    Base scheme & $S$ & $k$   \\ \hline
    Derived affine schemes & $\DGA_S$ & $\CDGA_k$   \\ \hline
    Derived Poisson schemes  & $\DGPPA_S $ & $\CDGPA_k$ \\ \hline
    Gauge algebra & $T_S(L)$ & $\Symg$ \\ \hline
    Hamiltonian spaces & $\PPA_{T_S(L)}^H\subset \DGPPA_{T_S(L)}$ & $\CPA_{\Symg}^H\subset\CDGPA_{\Symg}$ \\ \hline
    Derived zero locus (Koszul complex) & $A \amalg_{T_S(L)}^\L S \cong A \amalg_S T_S(L[1])$ &$B\otimes_{\Symg}^\L k \cong B\otimes_k \Sym(\lieg[1])$ \\\hline
   Quotient stack (Chevalley-Eilenberg complex) & $T_S(L^\ast[-1]) \amalg_{S} - $& $ \Hom_k(\Sym(\lieg[1]),-)$ \\\hline
   Derived Poisson reduction (BRST complex) & $T_S(L^\ast[-1]) \amalg_{S} \left(A \amalg_{T_S(L)}^\L S  \right) $& $ \Hom_k \left(\Sym(\lieg[1]),B\otimes_{\Symg}^\L k\right)$ \\\hline
    \end{tabular}
\end{center}
\end{figure}

After the theoretical part, we show the details of noncommutative derived Poisson reduction for some concrete well-known algebras such as cotangent bundles of smooth algebras, and in particular path algebras of doubled quivers, obtaining a BRST model for Nakajima-type quiver varieties. We give a proof of the somewhat classical result that the (commutative) BRST homology is the tensor product of the $\GL_\nn$-invariant homology of the corresponding Koszul complex with the Lie algebra (co)homology of $\gl_\nn$:

\begin{Thm*}[\S\ref{5.1}]
Let $A$ be a noncommutative Hamiltonian space and, for a fixed dimension $\nn$, let $\B_\nn(A),\C_\nn(A)$ the associated (commutative) BRST and Koszul complexes, respectively. Then we have
\begin{equation}
\HH_\bullet (\B_\nn(A)) \cong \HH_\bullet(\C_\nn(A))^{\GL_\nn} \otimes_k \HH^{-\bullet} (\gl_\nn,k)\, .
\end{equation}
\end{Thm*}

Then we show a couple of examples of path algebras of quivers such as the quiver with one vertex and $g$ loops, corresponding to a Lie algebra version of the character variety of a Riemann surface of genus $g$, and in particular the commuting scheme for $g=1$. Finally we pick two more examples different from path algebras of a quiver, which correspond to a Lie group-Lie algebra and a Lie group-Lie group version of the commuting scheme (which would be the case Lie algebra-Lie algebra).

\subsection{Layout of the paper and instructions for the reader} 
\label{1.2}
\S\ref{2} explains the theory of double Poisson algebras (as introduced by M. Van den Bergh), with a particular emphasis on the differential graded case and a categorical meaning of these structures. In \S\ref{3} we formalise the constructions contained in the noncommutative side of the `dictionary' (Figure~\ref{fig:dic}) together with a few structural results that these definitions are indeed well-posed and well-behaved. In \S\ref{4} we prove our main results, showing in which sense the noncommutative side of the `dictionary' corresponds to the commutative side. In \S\ref{5} we discuss the large class of examples of cotangent bundles, in particular path algebras of doubled quivers (together with a computation of the commutative BRST homology in this context) and various versions of the commuting scheme. In the last, short \S\ref{6} we introduce the notions of noncommutative analogues of more general group schemes, group actions and Poisson-group schemes, which are a possible direction in which we can generalise the results of the paper.

\subsection*{Notations and conventions} $k$ denotes an algebraically closed field of characteristic zero. We denote categories by the standard monospace font: $\Vect_k$, $\Alg_k$, $\DGA_k$, \dots We always work in chain complexes, so for us a differential graded object (algebra, vector space, \dots) has differential of degree $-1$, differential graded is often shortened by ``dg'', and commutative differential graded by ``cdg''. For a category $\CC$ and an object $S \in \CC$ we denote by $S \downarrow{\CC}$ the under category (in the case of dg algebras we denote this also by $S\downarrow \DGA_k=\DGA_S$). The coproduct in the under category $S\downarrow \CC$ is the push-out in $\CC$ of diagrams $\bullet \leftarrow S \rightarrow \bullet $, and denoted by $-\amalg_S - $. Left and right derived functors of a functor between model category $F: \CC \to \DD$, when they exist, are denoted by $\L(F), \R(F)$ and by them we mean the total left/right derived functors between the homotopy categories. In the case of schemes, we denote the derived pull-back also by the more traditional symbol: $X\times_Z^\R Y=X \times_Z^h Y$ (`h' stands for `homotopy' pull-back).

\subsection*{Acknowledgments} 
I want to thank my advisor G. Felder for our numerous discussions, his useful questions and for having shared with me his expertise in the BV-BFV-BRST formalism. I also want to thank Y. Berest for having first explained to me what is noncommutative Poisson geometry and for having asked some questions that raised my interest for the subject, and F. Naef for related discussions.

This work was supported by the National Centre of Competence in
Research SwissMAP --The Mathematics of Physics-- of the Swiss National
Science Foundation.


\section{Double Poisson algebras}
\label{2}

In~\cite{VdB} M. Van den Bergh introduced double Poisson brackets as the main candidates for noncommutative Poisson structures according to the Kontsevich-Rosenberg principle (indeed if one wants a Poisson bracket on representation schemes, needs a bracket with values in $A \otimes A$). In this Section we recall the main definitions and results from M. Van den Bergh in order to set up the notation and adapt them slightly to better suit our purposes. Mainly we discuss the differential graded version of double Poisson brackets (which is already sketched in~\cite{VdB}, and studied in~\cite{FH} in relation to cyclic $A_\infty$-algebras) and by doing so we consider the category of dg double Poisson algebras, which is the natural category in which we should do derived Poisson reduction. Finally we introduce a special class of dg double Poisson algebras whose differential is given by the induced single Poisson bracket with a distinguished (double) Maurer-Cartan element which we call ``noncommutative charge'', because the induced differential on representation schemes is obtained as the Poisson bracket by the trace of this element. The expert reader can skip this Section entirely, or just come back when some notion from this Section is used in the following part of the paper.


\subsection{Graded objects}
\label{2.1}

Let $\DGA_k$ be the category of differential graded algebras over $k$ (the differentials have degree $-1$ in our conventions). We recall that for a differential graded algebra $A$, the tensor product $A\otimes A$ has two natural graded bimodule structures over $A$:
\begin{itemize}
\item[(outer)] $a \cdot (u\otimes v) \cdot b :=au \otimes vb$ , 
\item[(inner)] $a \ast (u\otimes v) \ast b:= (-1)^{|a||b|+|a||u|+|b||v|} ub \otimes av$ .
\end{itemize}

The two structures commute with each other (with a sign):
\begin{equation}
\label{eq:out/inn}
\begin{cases}
a \cdot (a' \ast (u\otimes v) \ast b') \cdot b= (-1)^{|a||a'|+ |b||b'|} a' \ast (a \cdot (u \otimes v) \cdot b) \ast b'\,,\\
a' \ast (a \cdot (u \otimes v) \cdot b) \ast b' = (-1)^{|a||a'|+ |b||b'|} a\cdot (a' \ast (u\otimes v) \ast b') \cdot b\,.
\end{cases}
\end{equation}

For each $n=1,2,\dots $ and each permutation $\sigma \in \Sigma_n$ we denote by $\tau_\sigma :A^{\otimes n} \to A^{\otimes n}$ the isomorphism:
\[
\tau_\sigma(a_1\otimes \cdots \otimes a_n) = (-1)^s a_{\sigma^{-1}(1)}\otimes \cdots \otimes a_{\sigma^{-1}(n) } \,,
\]
where $s$ is the sign that counts all the swappings involved in $\sigma$:
\[
s = \sum\limits_{\substack{i < j:\\ \sigma^{-1}(i) >\sigma^{-1}(j)}} |a_{\sigma^{-1}(i)}||a_{\sigma^{-1}(j)}|
\, .\]
The permutation $(-)^\circ:= \tau_{(12)}:A^{\otimes 2} \to A^{\otimes 2}$ intertwines the two bimodule structure:
\begin{equation}
\label{eq:swapbim}
\begin{cases}
(a \cdot (u\otimes v) \cdot b)^\circ = a\ast ( (u \otimes v)^\circ )\ast b\,,\\
( a\ast  (u \otimes v) \ast b )^\circ = a \cdot ( (u\otimes v)^\circ) \cdot b \,.
\end{cases}
\end{equation}

For a graded object $A$ (a dg algebra, a graded vector space, \dots) and an integer $m \in \Z$ we denote its $m$-shifted object by $A[m]$:
\begin{equation}
\label{eq:shift}
(A[m])_i := A_{i-m} \, ,
\end{equation}
so that if $A$ is concentrated in degree zero, then $A[m]$ is concentrated in degree $m$, and a homogeneous map $A \to A[m]$ is a map $A_i \to A_{i-m}$ (shifted of degree $-m$).


\subsection{Multi-brackets on differential graded algebras}
\label{2.2}

\begin{Defn} 
\label{def:nbracket}
An $n$\textit{-bracket} on a differential graded algebra $A\in \DGA_k$ is a map $\ldb-,\dots,- \rdb:A^{\otimes n} \to A^{\otimes n}$ (of degree $0$) with the following properties:
\begin{enumerate}
\item (derivation) The map $\ldb a_1,\dots,a_{n-1}, -\rdb : A \to A^{\otimes n}$ is a graded derivation (for the outer bimodule structure on $A^{\otimes n}$) of degree $p:=|a_1|+\dots + |a_{n-1}|$, that is 
\begin{equation}
\label{eq:(1)}
\ldb a_1,\dots,a_{n-1}, bc\rdb = \ldb a_1,\dots,a_{n-1}, b\rdb \cdot c + (-1)^{|b|p} b\cdot \ldb a_1,\dots,a_{n-1}, c\rdb \,,
\end{equation}
\item (cyclic invariance)
\begin{equation}
\label{eq:(1)}
\ldb -, \dots ,- \rdb = (-1)^{n+1} \tau_{(12\dots n)} \circ \ldb -,\dots,-\rdb \circ \tau_{(12\dots n)}^{-1}\, ,
\end{equation}
\item (compatibility between bracket and differential) 
\begin{equation}
\label{eq:(1)}
d \circ \ldb -,\dots,-\rdb = \ldb -,\dots,-\rdb \circ d \, .
\end{equation}
\end{enumerate}
\end{Defn}
\begin{Defn}
\label{def:slinear}
If $A\in \DGA_S:=S\downarrow \DGA_k$ is a dg algebra over $S$, then an $n$-bracket is called $S$\textit{-linear} if it vanishes when its last argument is in the image of $S$ under the structure map $S\to A$ (and consequently, by the cyclic invariance, if any argument is in the image of $S$). 
\end{Defn}
\begin{Remark}
\label{rem:shifted}
There is a more general notion of $(-m)$-shifted $n$-bracket which is a bracket $\ldb -,\dots,-\rdb: A^{\otimes n} \to A^{\otimes n}[m]$ which satisfies the corresponding shifted properties analogous to (1),(2),(3) (see \cite{FH}). All the results in this Section hold also for $m$-shifted brackets, however, in this paper we do not need these structures, therefore we discuss only the $0$-shifted (homogeneous) case, which shortens the length of the signs involved in the formulas.
\end{Remark}

In the particular cases $n=2$ and $n=3$ such a structure is called, respectively, a double or a triple bracket. A double bracket is a map that satisfies 
\begin{enumerate}
\item $\ldb a,bc \rdb = \ldb a,b\rdb \cdot c + (-1)^{|a||b|} b \cdot \ldb a,c\rdb $ ,
\item $ \ldb a,b\rdb = - (-1)^{|a||b|} \ldb b,a \rdb^\circ$ ,
\item $d \ldb a,b\rdb = \ldb da,b \rdb + (-1)^{|a|} \ldb a, db \rdb $ ,
\end{enumerate}
and, because of~\eqref{eq:swapbim}, once property (1) is fixed, property (2) is equivalent to ask that the bracket is a (graded) derivation in the first argument, for the inner bimodule structure:
\begin{center}
($2^\ast$)\, $\ldb ab, c \rdb = a \ast \ldb b,c \rdb + (-1)^{|b||c|} \ldb a,c\rdb \ast b $ .
\end{center}
Given a binary operation $\ldb -,-\rdb$ (which does not have to be necessarily a double bracket) we define the following operation $A^{\otimes 3} \to A^{\otimes 3}$:
\begin{equation}
\label{eq:tripleL}
\ldb a, u\otimes v \rdb_L := \ldb a,u \rdb \otimes v \,,
\end{equation}
and using this we define the following triary operation 
\begin{equation}
\label{eq:triple}
\ldb a,b,c\rdb := \ldb a, \ldb b,c\rdb \rdb_L + (-1)^{|a|(|b|+|c|)} \tau_{(123)}\ldb b, \ldb c,a\rdb \rdb_L +(-1)^{|c|(|a|+|b|)}  \tau_{(132)} \ldb c, \ldb a,b\rdb \rdb_L\,, 
\end{equation}
or more abstractly
\begin{equation}
\label{eq:tripleab}
\ldb-,-,-\rdb = \sum\limits_{i=0}^2 \tau_{(123)}^i \circ \ldb -,\ldb-,-\rdb\rdb_L \circ \tau_{(123)}^{-i}\, ,
\end{equation}
which makes it clear that it is cyclically invariant. Moreover one can prove that if $\ldb-,-\rdb$ is a double bracket, then it is also a graded derivation in its last argument, so that:
\begin{Lem}[\cite{VdB}]
If $\ldb-,-\rdb$ is a double bracket then the associated triary operation $\ldb-,-,-\rdb$ is a triple bracket.
\end{Lem}
\begin{proof}
Let us show the super-derivation property in its last argument. The same calculations of the ungraded version of this Lemma (\cite[Proposition 2.3.1]{VdB}), if we keep track of signs, yield for the three summands of $\ldb a,b,cc'\rdb$:
\[
\ldb a,\ldb b,cc'\rdb\rdb_L + (-1)^{s_1} \tau_{(123)} \ldb b,\ldb cc',a \rdb \rdb_L+ (-1)^{s_2} \tau_{(132)} \ldb cc',\ldb a,b\rdb \rdb_L=: \circled{1} + \circled{2} + \circled{3}\,,
\]
where $s_1= |a|(|b|+|c|+|c'|)$ and $s_2=(|c|+|c'|)(|a|+|b|)$.
\[
\begin{aligned}
&\circled{1}=\ldb a,\ldb b,c\rdb\rdb_L \cdot c' + (-1)^{|b||c|}\ldb a,c\rdb \cdot \ldb b,c \rdb + (-1)^{(|a|+|b|)|c|} c\cdot \ldb a,\ldb b,c'\rdb \rdb_L\,,\\
&\circled{2}= (-1)^{s_1+|b||c|}c\cdot \tau_{(123)}\ldb b,\ldb c',a\rdb\rdb_L + \\&(-1)^{s_1+ |a||c'|}\tau_{(123)}\ldb b,\ldb c,a\rdb \rdb_L \cdot c' +
-(-1)^{|b||c|} \ldb a,c\rdb\cdot \ldb b,c'\rdb \,, \\
&\circled{3}=(-1)^{s_2} c\cdot \tau_{(132)} \ldb c',\ldb a,b\rdb \rdb_L + (-1)^{|c|(|a|+|b|)} \tau_{(132)}\ldb c,\ldb a,b\rdb \rdb_L\cdot c' \,,
\end{aligned}
\]
where, in lines 1 and 2, by $(x\otimes y) \cdot (u \otimes v)$ we mean $x \otimes yu\otimes v$. Summing the three expressions we obtain
\[
\ldb a,b,cc'\rdb = \ldb a,b,c\rdb \cdot c' + (-1)^{(|a|+|b|)|c|}c\cdot \ldb a,b,c'\rdb \,.
\]
As for the compatibility between $\ldb-,-,-\rdb$ and the differential, this follows from the fact that both
\[
\ldb-,\ldb -,-\rdb\rdb_L= (1_A\otimes \ldb-,-\rdb) \circ (\ldb-,-\rdb \otimes 1_A) 
\]
and $\tau_\sigma$, for any permutation $\sigma$, commute with the differential.
\end{proof}


\subsection{Double Poisson brackets}
\label{2.3}

\begin{Defn}[\cite{VdB}] A double bracket $\ldb -,-\rdb$ on a dg algebra $A\in \DGA_S$ is called a \textit{double Poisson bracket} if the associated triple bracket is zero: $\ldb -,-,-\rdb = 0$. The identity $\ldb a,b,c\rdb=0$ is called (graded) double Jacobi identity. The pair $(A,\ldb-,-\rdb)$ is called a \textit{differential graded double Poisson algebra}.
\end{Defn}

\begin{Defn} A \textit{morphism} of differential graded double Poisson algebras is a morphism $\varphi:A \to B$ of dg algebras over $S$, such that the induced map $\varphi^{\otimes2}: A^{\otimes 2} \to B^{\otimes2}$ intertwines the two double Poisson structures. We denote the (so obtained) category of dg double Poisson algebras by $\DGPPA_S$.
\end{Defn}

\begin{Remark} Differently from $\DGA_S = S\downarrow \DGA_k$ which denotes the under category, the category of dg double Poisson algebras over $S$ is not the under category $\DGPPA_S \neq S\downarrow \DGPPA_k$ (with $S$ equipped with the zero double Poisson structure). In fact an object of the former has the property that the bracket with any element in the image of $S$ vanishes, while for an object of the former the bracket vanishes a priori only if both variables belong to the image of $S$.
\end{Remark}

\begin{Notation} We denote the full subcategory of dg double Poisson algebras consisting of algebras concentrated in degree zero by $\PPA_S \subset \DGPPA_S$.
\end{Notation}

Let us denote the multiplication map by $m: A^{\otimes 2} \to A$ and, given a double bracket $\ldb-,-\rdb$, let us consider the associated single bracket 
\[
\{-,-\}: A^{\otimes 2} \xrightarrow[]{\ldb-,-\rdb} A^{\otimes 2} \xrightarrow[]{m} A\, .
\]
\begin{Prop}[\cite{VdB}]
\label{prop:almostJacobi}
If $\ldb-,-\rdb$ is a double bracket then the following equation holds in $A^{\otimes 2}$:
\begin{equation}
\begin{aligned}
\label{eq:almostJacobi}
&\{ a,\ldb b,c\rdb \} - \ldb \{a,b\}, c\rdb - (-1)^{|a||b|} \ldb b,\{a,b\} \rdb =\\
&= (m\otimes 1) \ldb a,b,c\rdb -(-1)^{|a||b|} (1\otimes m) \ldb b,a,c\rdb \, ,
\end{aligned}
\end{equation}
where $\{a,-\}$ acts on tensors $u\otimes v$ by $\{a,u\otimes v\} =\{a,u\} \otimes v +(-1)^{|a||u|} u \otimes \{a,v\}$.
\end{Prop}
\begin{proof}
With a few intermediate calculations one proves that
\[
\begin{aligned}
&\{a,\ldb b,c\rdb\} = (m \otimes 1) \ldb a,\ldb b,c\rdb \rdb_L + (-1)^{|b||c|} (1\otimes m) \tau_{(123)} \ldb a,\ldb c,b\rdb \rdb_L\,,\\
&\ldb \{ a,b \}, c\rdb =- (-1)^{|c|(|a|+|b|)} (m \otimes 1) \tau_{(132)} \ldb c,\ldb a,b \rdb \rdb_L + \\
&\qquad \qquad \qquad (-1)^{|a||b|+ |b||c| + |a||c|}(1\otimes m)  \tau_{(132)} \ldb c,\ldb b,a\rdb \rdb_L\,,\\
&\ldb b, \{a,c\} \rdb =- (-1)^{|a||c|} (m \otimes 1) \tau_{(123)} \ldb b,\ldb c,a \rdb \rdb_L + (1\otimes m)  \ldb b,\ldb a,c\rdb \rdb_L\,,\\
\end{aligned}
\]
from which equation~\eqref{eq:almostJacobi} follows.
\end{proof}

Let us denote by $[A,A] \subset A $ the linear subspace spanned by graded commutators and the quotient by $A_\n =A/[A,A]$. For an element $a \in A$ we denote by $\overline{a} \in A_\n$ its class modulo $[A,A]$.

\begin{Lem}[\cite{VdB}]
\label{lem:singlebracket}
Let $\ldb-,-\rdb$ be a double bracket. Then the associated single bracket $\{-,-\}:A^{\otimes 2} \to A$ has the following properties:
\begin{enumerate}
\item $\{ [A,A], - \} =0$ ,
\item $\{a,-\} $ is a graded derivation of degree $|a|$ ,
\item $\overline{\{ a,b\}} = -(-1)^{|a||b|} \overline{\{b,a\} } $ ,
\item $d \circ \{-,-\} = \{ -,-\} \circ d $ ,
\item If $\ldb-,-\rdb$ is a double Poisson bracket, then the following ``Leibniz property'' (a version of the Jacobi identity) holds in $A$:
\begin{equation}
\label{eq:Leibniz}
\{ a,\{b,c\} \} = \{ \{ a,b\} ,c\} + (-1)^{|a||b|} \{ b,\{a,c\} \}\, .
\end{equation}
\end{enumerate}
\end{Lem}
\begin{proof}
Using the super-derivation property in the first argument:
\[
\ldb [a,b] ,c\rdb = a \ast \ldb b,c\rdb -(-1)^{|a|(|b|+|c|)}\ldb b,c\rdb \ast a + (-1)^{|b||c|} ( \ldb a,c\rdb \ast b - (-1)^{|b|(|a|+|c|)} b \ast \ldb a,c\rdb )\, ,
\]
and in general for $a \in A$ and $\omega \in A^{\otimes2}$: $m(a \ast \omega) = (-1)^{|a||\omega|} m (\omega \ast a)$, from which (1) follows: 
\[
\{[a,b],c\} = m \ldb [a,b],c\rdb = 0\, .
\]
(2) is obvious. For (3) it is enough to observe that for any $\alpha \in A^{\otimes2}$, $m (\alpha-\alpha^\circ) \in [A,A]$, and apply this to $\alpha=\ldb a,b\rdb$. (4) is:
\[
d\circ \{-,-\} = d\circ m \circ \ldb-,-\rdb = m\circ d\circ\ldb -,-\rdb =m \circ \ldb -,-\rdb \circ d = \{-,-\} \circ d\, ,
\]
where by the same symbol $d$ we mean both the differential on $A$ and the induced one on $A^{\otimes 2}$. (5) follows from applying the multiplication map to~\eqref{eq:almostJacobi}.

\end{proof}

By (1) and (2) the map $\{-,-\}$ induces a well-defined map $\{-,-\}_\n : A_\n^{\otimes 2} \to A_\n$ which, by properties (3) (4) and (5), makes $A_\n$ into a differential graded Lie algebra $(A_\n, \{-,-\}_\n) \in \DGLA_k$:

\begin{Lem}
If $(A,\ldb -,-\rdb)$ is a double Poisson algebra, the induced bracket $\{-,-\}_\n$ on $A_\n$ makes it a differential graded Lie algebra.
\end{Lem}

\begin{proof}
Indeed if we use antisymmetry of the induced bracket on $A_\n$, the Leibniz identity~\eqref{eq:Leibniz} becomes the (graded) Jacobi identity in its usual form:
\[
(-1)^{|a||c|}\{ a,\{b,c\}_\n \}_\n + (-1)^{|a||b|} \{ b,\{c,a\}_\n \}_\n + (-1)^{|b||c|} \{ c, \{ a,b\}_\n \}_\n =0\, .
\]
The compatibility between the differential and $\{-,-\}_\n$ follows from the compatibility between the differential and $\{-,-\}$ and the projection $A\to A_\n$.
\end{proof}

The bracket $\{-,-\}_\n$ is slightly more than simply a (dg) Lie structure, in fact for each element $a \in A$, the map $\{\overline{a},-\}_\n: A_\n \to A_\n$ is induced by a (graded) derivation (of degree $|a|$) $\partial_a = \{a,-\}$. This is what is called a (differential graded) $\HH_0$-Poisson structure on $A$:

\begin{Defn}[\cite{Cr} --- ungraded version]
A (differential graded) $\HH_0$\textit{-Poisson structure} on a (differential graded) algebra $A\in \DGA_S$ is a (differential graded) Lie bracket $\{-,-\}_\n$ on $A_\n$ with the property that for each homogeneous element $a \in A$, the map $\{\overline{a},-\}_\n$ is induced by a graded, $S$-linear derivation $\partial_a : A \to A$ of degree $|a|$. We call the pair $(A,\{-,-\}_\n)$ an $\HH_0$\textit{-Poisson algebra (over $S$)}.
\end{Defn}

\begin{Defn}
A \textit{morphism} of $H_0$-Poisson algebras is a morphism $\varphi:A \to B$ of dg algebras over $S$ such that the induced map $\varphi_\n : A_\n \to B_\n$ is a morphism of dg Lie algebras. We denote the (so obtained) category of dg $\HH_0$-Poisson algebras by $\DGPA_S$.
\end{Defn}

\begin{Lem}
\label{lem:H0P}
There is a natural forgetful functor $\DGPPA_S \to \DGPA_S$ which sends a double Poisson algebra $(A,\ldb-,-\rdb)$ to the $\HH_0$-Poisson algebra $(A,\{-,-\}_\n)$, where $\{-,-\}_\n$ is the bracket on $A_\n$ induced by the single bracket associated to $\ldb-,-\rdb$.
\end{Lem}

\begin{Remark} 
\label{rem:H0CB}
$\HH_0$-Poisson structures were first introduced by W. Crawley-Boevey in \cite{Cr}. They are the natural structure to consider if one wants an induced ordinary Poisson structure on the $\GL_n$-invariant part of the representation scheme $A_n^{GL_n}$. In fact one proves that there is only one induced Poisson structure on $A_n^{GL_n}$ with the property that the trace map $\tr :A_\n \to A_n^{GL_n}$ is a map of dg Lie algebras:
\begin{equation}
\label{eq:trace}
 \tr \{ a ,b\} = \{ \tr(a),\tr(b)\}\,.
\end{equation}
The reason why such a `single' noncommutative Poisson structure is enough if one wants a Poisson structure on the $\GL_n$-invariant subalgebra is that the latter is generated by traces, therefore actually the Poisson structure depends only on $2$ indices, and not $4$. If one wants a Poisson structure on the whole $A_n$, is forced to consider double Poisson structures.
\end{Remark}


\subsection{Building new double Poisson structures from old}
\label{2.4}

\begin{Prop}[\cite{VdB}]
\label{prop:free}
If $(A,\ldb-,-\rdb_A)$ and $(B,\ldb-,-\rdb_B)$ are double Poisson algebras over $S$, their free product $A\ast_S B$ (coproduct in the category of dg algebras over $S$) has an induced natural double Poisson structure over $S$, defined uniquely by the formulas:
\begin{equation}
\label{eq:coprod}
\ldb a, b\rdb := 0\,,\quad \ldb a,a'\rdb:=\ldb a,a'\rdb_A\,,\quad \ldb b,b'\rdb :=\ldb b,b'\rdb_B\,,
\end{equation}
for each $a,a' \in A$ and $b,b' \in B$.
\end{Prop} 
\begin{proof} As the ungraded version of this Proposition (\cite[Proposition 2.5.1]{VdB}) the proof is left to the reader, because fairly easy. In fact the induced structure on the coproduct does not mix the two structures, therefore its properties (super-derivation, cyclic invariance, compatibility with the differential) follow from the corresponding properties of the double Poisson brackets on $A$ and $B$.
\end{proof}

\begin{Prop}
\label{prop:cop}
The free product construction in Proposition~\ref{prop:free} is the coproduct in the category $\DGPPA_S$. The algebras $S$ and $0$, both with zero double Poisson structure, are respectively, the inital and final object in $\DGPPA_S$.
\end{Prop}
\begin{proof} Let $C \in \DGPPA_S$ and $\varphi : A \to C$, $\psi:B \to C$ morphisms of $\DGPPA_S$. Because morphisms of dg double Poisson are morphisms of dg algebras with additional properties, in particular we have a unique morphism of dg algebras $F: A \ast_S B \to C$ that extends $\varphi$ and $\psi$. The fact that $F^{\otimes 2}$ commutes with the double Poisson bracket defined in~\eqref{eq:coprod} can be easily tested on the generators:
\[
\begin{aligned}
&F^{\otimes 2}(\ldb a ,b \rdb ) = 0 = \ldb \varphi(a),\psi(b) \rdb = \ldb F(a),F(b)\rdb\,, \\
&F^{\otimes 2}(\ldb a,a' \rdb ) = \varphi^{\otimes 2} (\ldb a,a' \rdb_A) = \ldb \varphi(a),\varphi(a') \rdb_A = \ldb F(a), F(a')\rdb \,,\\
&F^{\otimes 2} (\ldb b,b'\rdb) = \psi^{\otimes 2} (\ldb b,b'\rdb_B) = \ldb \psi(b),\psi(b')\rdb_B = \ldb F(b),F(b')\rdb \,.
\end{aligned}
\]
The fact that $S$ and $0$ are the initial and final objects in $\DGPPA_S$ follows from the fact that they are such objects in $\DGA_S$ and that for any $A\in \DGPPA_S$ the initial and terminal map $S\to A \to 0$ are maps of double Poisson algebras (because of the assumption of $S$-linearity on the double Poisson bracket on $A$).
\end{proof}

Another construction that will be useful later is the particular case of differential graded algebras $A$ with a double Poisson bracket and such that their differential is of the very special form 
\[
d= \{ \gamma , - \} 
\]
where $\{-,-\}$ is the associated single bracket on $A$. If we start simply from a graded algebra with a double Poisson bracket without differential we give the following definition:
\begin{Defn}
\label{def:charge}
Let $A$ be a graded algebra and $\ldb-,-\rdb$ a double Poisson bracket on it. We call an element $\gamma$ of degree $|\gamma|=-1$ a \textit{noncommutative charge} if the single bracket of it by itself lies in the (graded) commutators subspace: $\{\gamma,\gamma\} \in [A,A]$.
\end{Defn}

\begin{Prop}
\label{prop:charge}
If $\gamma$ is a noncommutative charge on $(A,\ldb-,-\rdb)$, then $d=\{\gamma,-\}$ is a differential on $A$ which is compatible with the double Poisson bracket, so it gives $A$ the structure of a differential graded double Poisson algebra.
\end{Prop}

\begin{proof}
Obviously $d$ is a linear map of degree $-1$ satisfying $d(ab) =d(a)b +(-1)^{|a|}a d(b)$ (from Lemma~\ref{lem:singlebracket}, (2)). If we apply the Leibniz property (Lemma~\ref{lem:singlebracket}, (5)) to $a=b=\gamma$ we obtain:
\[
d^2 = \frac12 \{ \{\gamma,\gamma\}, -\} \,,
\]
so that $d^2=0$ follows from $\{\gamma,\gamma\} \in [A,A]$ (because of Lemma~\ref{lem:singlebracket}, (1)). From Proposition~\ref{prop:almostJacobi} applied to $a=\gamma$ we have 
\[
d \ldb b,c\rdb = \ldb db, c \rdb + (-1)^{|b|} \ldb b,dc\rdb \,,
\]
which is the desired compatibility between the bracket and the differential.
\end{proof}


\section{Derived noncommutative Poisson reduction}
\label{3}

The aim of this Section is to introduce the noncommutative constructions mentioned in the left side of the `dictionary' (Figure~\ref{fig:dic}), together with some proofs of a few structural results that these definitions are well-behaved and satisfy some good properties. The impatient reader who wants to know why we give such definitions is encouraged to jump from time to time from this Section to \S\ref{4}, which hopefully clarifies everything.


\subsection{Crash course in noncommutative geometry}
\label{3.1}

This first Section is a crash course in noncommutative geometry. We recall a few definitions and basic results in the theory, but we try to stick to the minimum required in order to understand the following parts of the paper. There is obviously much more to be said, and the interested reader can consult the more foundational references \cite{CEG,G2,G1}.

Throughout this Section we consider algebras $A\in \Alg_S$ for which the structure map $i:S\to A$ is injective. The space of noncommutative $1$-forms, or K{\"a}hler differentials, is the $A$-bimodule $\Omega_S^1 A:=\ker(m)$, kernel of $m : A\otimes_S A \to A$, the multiplication map (over $S$). The universal ($S$-linear) derivation is the map $d\in \Der_S(A,\Omega_S^1 A)$ defined by 
\begin{equation}
\label{eq:univder}
d(a) =a \otimes 1 - 1 \otimes a\, .
\end{equation}
The couple $(\Omega_S^1 A, d)$ is a universal $A$-bimodule equipped with a derivation in the sense that for each $A$-bimodule $M$, there is a natural isomorphism
\begin{equation}
\label{eq:univadj}
\Der_S(A,M) \xrightarrow[]{\sim} \Hom_{A-\Bimod}(\Omega_S^1 A,M)\,, \quad \partial \mapsto \varphi_\partial \,,
\end{equation}
where $\varphi_\partial$ is the only morphism such that $\varphi_\partial \circ d = \partial$. Because $\Omega_S^1 A$ is spanned by the elements $a d b$, for $a,b \in A$, we have $\varphi_\partial(a db) = a \partial(b)$. 

Let us consider the controvariant (duality):
\begin{equation}
(-)\dual := \Hom_{A-\Bimod} (-,A\otimes A): A-\Bimod \to A-\Bimod\, , 
\end{equation}
where we view $A\otimes A$ with the outer $A$-bimodule structure, but the inner structure survives and gives, for each $M\in A-\Bimod$, the above-claimed $A$-bimodule structure on $M\dual$. In particular if we consider $M=\Omega_S^1 A$ we obtain, by~\eqref{eq:univadj}, a natural isomorphism between the noncommutative vector fields $\Theta_S^1 A:=(\Omega_S^1A)\dual$ and the space of double ($S$-linear) derivations:
\begin{equation}
\label{eq:ncvectorfields}
\DDer_S(A) := \Der_S(A,A\otimes A) \cong \Hom_{A-\Bimod}(\Omega_S^1 A,A\otimes A) =: \Theta_S^1 A
\end{equation}

\begin{Defn} The \textit{noncommutative cotangent bundle} of $A$ is the tensor algebra of the $A$-bimodule of noncommutative vector fields, and it is denoted by
\begin{equation}
\label{eq:nccot}
T^\ast A : = T_A \left( \Theta_S^1 A \right) = T_A \left( \DDer_S ( A) \right) \, .
\end{equation}
\end{Defn}

\begin{Remark}
\label{rem:shiftedcot}
This could be considered either as graded algebra or as an ungraded algebra placed all in degree zero. In the first case if we apply the representation functor to it we obtain the shifted cotangent bundle of the representation scheme of $A$, and in the second case the ordinary (unshifted) cotangent bundle. In this paper we consider mainly the second version (the ungraded one), for which we reserve the symbol $T^\ast A( =T_A \DDer_S(A))$, while when we really want to specify the graded version, we write $T^\ast_\odd A(=T_A \DDer_S(A)[1])$. 
\end{Remark}

The setting in which later we do noncommutative Poisson reduction is the case in which the algebra $S$ is the vector space generated by a finite set $S=kI$, consisting of orthogonal idempotents: $e_i e_j=\delta_{ij} e_j$. In this case we can define the following distinguished double derivation $\Delta \in \DDer_S(A)$:
\begin{equation}
\label{eq:distdelta}
\Delta(a) = \sum\limits_i ( a e_i\otimes e_i -e_i\otimes e_i a)\,,
\end{equation}
which is a preferred lifting of the universal derivation in the sense that $d=\pi\circ \Delta$, where $\pi: A\otimes A \to A\otimes_S A$ denotes the canonical projection. This distinguished double derivation can be divided into its components over $I$:
\begin{equation}
\label{eq:deltai}
\Delta= \sum\limits_i \Delta_i, \quad \Delta_i(a) = ae_i \otimes e_i-e_i\otimes e_i a\, .
\end{equation}


\subsection{Natural double Poisson structure on cotangent bundles}
\label{3.2}

M. Van den Bergh introduced a natural $(-1)$-shifted (see Remark~\ref{rem:shifted}) graded double Poisson structure on $T^\ast_\odd A$  which is called double Schouten-Nijenhuis structure (\cite[\S3.2]{VdB}). In this Section we recall his construction and we show that the same definition (on generators) actually gives ($0$-shifted) double Poisson structure on $T^\ast A$, which --- with a minus sign --- is the one we are interested in.

Throughout this Section we suppose that $A$ is finitely generated (all the examples we want to consider in this paper are of this form). In this case there is a natural identification between triple derivations and the tensor product of double derivations with $A$ itself:
\begin{equation}
\label{eq:3der}
\begin{aligned}
\DDer_S&(A) \otimes A \xrightarrow{\sim} \Der_S(A,A^{\otimes3} )  \\
& \partial \otimes a \longmapsto \Psi_{\partial \otimes a} (b) = \partial(b)' \otimes a \otimes \partial(b)''\, ,\end{aligned}
\end{equation}
where we use the (sumless) Sweedler notation for $\partial(b) = \partial(b)'\otimes \partial(b)''$. This follows from the identification between derivations and bimodules~\eqref{eq:univadj} and the finitely generatedness condition. Because of the form of~\eqref{eq:3der}, if we start from a triple derivation $\Psi \in \Der_S(A,A^{\otimes3})$ and we compose it with either $\tau_{(23)}$ or $\tau_{(12)}$ we obtain, respectively:
\begin{equation}
\label{eq:3derswap}
\begin{aligned}
\tau_{(23)} \circ \Psi  \in \DDer_S(A) \otimes A \, ,\\
\tau_{(12)} \circ \Psi \in A \otimes \DDer_S(A)\, .
\end{aligned}
\end{equation}
\begin{Prop} [{\cite[\S3.2]{VdB}}]
Let $\delta,\partial \in \DDer_S(A)$. Then 
\begin{equation}
\label{eq:dPlr}
\begin{aligned}
&\ldb \delta, \partial \rdb_l^\sim:= (\delta \otimes 1_A) \circ \partial - (1_A\otimes \partial) \circ \delta \,,\\
&\ldb \delta,\partial \rdb_r^\sim := (1_A \otimes \delta)\circ \partial - (\partial \otimes 1_A)\circ \delta = - \ldb \partial,\delta\rdb_l^\sim\,,
\end{aligned}
\end{equation}
define elements of $\Der_S(A,A^{\otimes 3})$.
\end{Prop}
When we compose them with the above-mentioned permutations (\eqref{eq:3derswap}) we can view them as elements of 
\begin{equation}
\label{eq:dPlrfinal}
\begin{aligned}
&\ldb \delta,\partial \rdb_l := \tau_{(23)} \circ \ldb \delta,\partial \rdb_r^\sim \in \DDer_S(A)\otimes A \subset (T^\ast A)^{\otimes2}\, ,\\
&\ldb \delta, \partial \rdb_r:= \tau_{(12)} \circ \ldb \delta,\partial \rdb_r^\sim \in A \otimes \DDer_S(A) \subset (T^\ast A)^{\otimes2}\, ,\\
\end{aligned}
\end{equation}
and they are related one to the other by swapping the first and the second component: $\ldb \delta,\partial\rdb_r = -\ldb \partial,\delta\rdb_l^\circ$.

\begin{Thm}[{\cite[Theorem 3.2.2]{VdB}}] 
\label{thm:SN}
There is a natural $(-1)$-shifted graded double Poisson structure (over $S$) on the noncommutative cotangent bundle $T^\ast_\odd A $ of a finitely generated algebra $A$, called double Schouten-Nijenhuis structure. It is defined uniquely by the formulas:
\begin{equation}
\label{eq:dPcot}
\ldb a,b\rdb = 0\,, \quad \ldb \delta, a \rdb = \delta(a)\, ,\quad \ldb \delta ,\partial \rdb =\ldb \delta,\partial\rdb_l +\ldb \delta,\partial\rdb_r  \,,
\end{equation}
for $a,b\in A$, $\delta,\partial \in \DDer_S(A)$.
\end{Thm}

\begin{Prop}
\label{prop:0shift}
The same definitions on the generators give a $0$-shifted double Poisson structure on the ungraded cotangent bundle $T^\ast A$, which --- with a minus --- we call the ``natural'' double Poisson structure:
\begin{equation}
\label{eq:dPcoteven}
\ldb a,b \rdb =0\,, \quad \ldb \delta, a \rdb =- \delta(a)\,, \quad \ldb \delta ,\partial \rdb= -\ldb \delta,\partial\rdb_l - \ldb \delta,\partial\rdb_r\,.
\end{equation}
\end{Prop}
\begin{proof} 
Let us consider the following general situation. Let $A\in \Alg_S$ and a bimodule $\D\in A-\Bimod$ consider its tensor algebra in the two following versions: odd $T_A ( \D[1])$ and even $T_A\D$. A $-1$-shifted double Poisson bracket on $T_A(\D[1])$ is uniquely determined by giving:
\[
\ldb a,b\rdb = 0\,,\quad \ldb \delta, a \rdb \in A^{\otimes 2}\,,\quad \ldb \delta,\partial \rdb \in \D\otimes A + A \otimes \D\, ,
\]
for $a,b\in A$, $\delta,\partial \in \D$ with the following properties:
\begin{enumerate}
\item $ \ldb \delta ,- \rdb : A \to A^{\otimes 2}$ is a double ($S$-linear) derivation in the second argument and it is compatible with the $A$-bimodule structure in the first argument.
\item $ \ldb \delta ,\partial \rdb = - \ldb \partial , \delta\rdb^\circ$ for each $\delta,\partial \in \D$.
\item The double Jacobi identity is satisfied on the generators, and this can be tested only in the following two situations:
\begin{itemize}
\item [(i)] two elements of $\D$ and one of $A$:
\[
\ldb \delta, \partial, a \rdb= \ldb \delta, \ldb \partial, a \rdb \rdb_L + \tau_{(123)}  \ldb \partial, \ldb a, \delta \rdb \rdb_L + \tau_{(132)} \ldb a , \ldb \delta, \partial \rdb \rdb_L=0\,.
\]
\item [(ii)] three elements of $\D$:
\[
\ldb \delta, \partial, d \rdb= \ldb \delta, \ldb \partial, d \rdb \rdb_L + \tau_{(123)}  \ldb \partial, \ldb d, \delta \rdb \rdb_L + \tau_{(132)} \ldb d , \ldb \delta, \partial \rdb \rdb_L=0\,.
\]
\end{itemize}
\end{enumerate}
The crucial observation is that the signs involved in the properties to be tested on the generators are the same as in the ungraded case (essentially because we only need to test them on particular choices of elements of degree zero and one). Therefore the same definitions actually give a ($0$-shifted) double Poisson structure on $T_A \D$ which coincides with the $-1$-shifted on $T_A (\D[1])$ on the generators. Now we apply this in the case of $\D = \DDer_S(A)$ with the natural $A$-bimodule structures and definitions~\eqref{eq:dPcot} which, by \cite[Theorem 3.2.2]{VdB}, define a $-1$-shifted double Poisson structure on $T^\ast_\odd A$.
\end{proof}


\subsection{Noncommutative Hamiltonian spaces}
\label{3.3}

In this Section we restrict to the case $S=kI$, a finite dimensional algebra made of orthogonal idempotents. We want to defined a noncommutative version of Hamiltonian spaces, and in order to do so we should first introduce a noncommutative version of the Lie algebra $\gl_\nn$.

We consider the path algebra of the quiver with vertices $I$ and with simply one loop $t_i$ on each vertex $i\in I$. It is the tensor algebra over $S$ of the $S$-bimodule $L=\Span_k\{t_i\}$ (isomorphic to $S$ itself, but obviously we call its basis elements with different names to distinguish them from the orthogonal idempotents $e_i \in I$ in the path algebra):
\begin{equation}
\label{eq:L}
T_S (L) =  \mathrm{PathAlg}_k \left( \begin{tikzcd} 1\arrow[out=0,in=90,loop,"t_1"] & 2\arrow[out=0,in=90,loop,"t_2"]  &\cdots  &{|I|}\arrow[out=0,in=90,loop,"t_{|I|}"]   \end{tikzcd} \right) .
\end{equation}
It has a natural double Poisson structure defined uniquely by \begin{equation}
\label{eq:k[t]}
\ldb t_i ,t_j\rdb =\delta_{ij} (  t_i \otimes e_i -e_i \otimes t_i )\, .
\end{equation}
The verification that~\eqref{eq:k[t]} actually defines a double Poisson is straightforward. 

As explained in the introduction the role of $T_S(L)$ in the noncommutative world is analogous to the one played by the Lie algebra of the gauge group $\gl_\nn =\Lie(\GL_\nn)$ acting on representation schemes (in fact the standard Poisson structure on $\Sym(\gl_\nn)$ is the one induced by the above-mentioned double Poisson structure). Any double Poisson algebra $A \in \PPA_S$ with a map $T_S(L) \to A$, has an induced map $\Sym(\gl_\nn) \to A_\nn$ via the representation functor, which is a Poisson morphism guaranteed that the map $T_S(L) \to A $ is a map of double Poisson algebra. In order to obtain a Hamiltonian space $A_\nn$, we need a compatibility condition between the $\gl_\nn$-action and the induced Poisson bracket on $A_\nn$, and this is given by the condition~\eqref{eq:6*}.

\begin{Defn}
\label{def:ncHamspa}
Let $A \in \PPA_S$ a double Poisson algebra over $S$. A \emph{noncommutative Hamiltonian action} on it is a morphism $f: T_S(L) \to A$ of double Poisson algebras, such that, for each $i \in I$ and for each $a \in A$:
\begin{equation}
\label{eq:hamilton}
\ldb f(t_i), a \rdb = a e_i \otimes e_i -e_i \otimes e_i a\, (=\Delta_i(a))\, .
\end{equation}
We call such a double Poisson $A$ equipped with a noncommutative Hamiltonian action a \emph{noncommutative Hamiltonian space}. 
\end{Defn}

\begin{Remark} This definition agrees with \cite[Definition 2.6.4.]{VdB} of \emph{Hamiltonian algebra}. The map $T_S(L) \to A$ is also called a \emph{noncommutative moment map}.
\end{Remark}

\begin{Defn} 
A \emph{morphism of noncommutative Hamiltonian spaces} is just a morphism of double Poisson algebras $A \to B$ compatible with the structure maps $T_S(L) \to A,B$. We denote the category of noncommutative Hamiltonian spaces by $\PPA_{T_S(L)}^H$.
\end{Defn}

\begin{Remark}
In other words the category of noncommutative Hamiltonian spaces can be seen as a full subcategory of the under category $\PPA_{T_S(L)}^H\subset \DGPPA_{T_S(L)}:=T_S(L)\downarrow \DGPPA_S$ (with an abuse of notation with respect to the previous notation $\DGPPA_S$ which was \emph{not} the under category of $\DGPPA_k$ with respect to $S$).
\end{Remark}

\begin{Notation}
From now on we will often denote the images of the loops under a Hamiltonian action $T_S(L) \to A$ by the symbols:
\begin{equation}
t_i \longmapsto f(t_i) =\delta_i  \in A\,,
\end{equation}
where the choice of the letter (small) `delta' is motivated by the fact that their Poisson brackets on $A$ give the action of the (big) deltas, the components of the chosen lifting of the universal derivation:
\begin{equation}
\ldb \delta_i , a \rdb =ae_i\otimes e_i - e_i \otimes e_i a = \Delta_i(a)\,.
\end{equation}
We also denote their sum by $\delta= \sum_i \delta_i$ and observe that a Hamiltonian action is nothing else but a choice of a diagonal element $\delta \in \oplus_i e_i A e_i$ such that
\begin{equation}
\label{eq:Hamiltoniandelta}
\ldb \delta, a \rdb = \sum_i ae_i\otimes e_i - e_i \otimes e_i a= \Delta(a)\, .
\end{equation}
The elements $\{\delta_i\}$ play the role of \emph{noncommutative gauge elements}.
\end{Notation}

If we consider the two-sided ideal generated by the image of the Hamiltonian action $\mathcal{I}=\langle \delta \rangle \subset A$, we obtain as a quotient the noncommutative analogue of the zero locus of the Hamiltonian map: $A/\mathcal{I}$. We can also view it as the coproduct $A/\mathcal{I} = A \amalg_{T_S(L)} S  \in \Alg_{T_S(L)}$, where $S$ is viewed as a $T_S(L)$-algebra through the obvious projection $T_S(L)\to S$ that sends $L$ to zero. 

\begin{Defn}
The \emph{noncommutative zero locus} of a Hamiltonian space $A \in \PPA_{T_S(L)}^H$ is the algebra $A/\mathcal{I}= A \amalg_{T_S(L)} S \in \Alg_{T_S(L)}$.
\end{Defn}

\begin{Remark}[Poisson structure on the zero locus]
\label{rem:PS}
As for the Poisson structure, the double Poisson structure on $A$ does not descend on the zero locus, however it is easy to verify that the following formula defines an induced $\HH_0$-Poisson structure on $A/\mathcal{I}$:
\[
\begin{aligned}
&\left( A\slash \mathcal{I} \right)_\n^{\otimes 2} \xrightarrow{\{-,-\}_{\n}} \left(A \slash \mathcal{I} \right)_\n\\
&\overline{a + \mathcal{I}} \otimes \overline{b +\mathcal{I}} \longmapsto \overline{ \{a,b\} + \mathcal{I}}
\end{aligned}
\]
where $\overline{a + \mathcal{I}} \in (A\slash \mathcal{I})_\n$ denotes the class modulo commutator of the element $a + \mathcal{I} \in A \slash\mathcal{I}$, and for $a,b \in A$, $\{a,b\} \in A$ is the single bracket associated to the double Poisson structure on $A$.
\end{Remark}

Following the construction done in the commutative setting we define a derived version of the zero locus by using the total left-derived functor of the coproduct (we refer the interested reader to Appendix~\ref{app:A} for explanations on the derived coproduct).

\begin{Defn}
The \emph{derived noncommutative zero locus} of a Hamiltonian space $A \in \PPA_{T_S(L)}$ is the (homotopy class of the) dg algebra
\begin{equation}
\label{eq:derzero}
A \amalg_{T_S(L)}^\L S \in \Ho(\DGA_{T_S(L)})\, .
\end{equation}
\end{Defn}

In what follows however, we will often consider a specific model for it, obtained by choosing the following cofibrant replacement of $S$ in the category $\DGA_{T_S(L)}$ (see Appendix~\ref{app:A} on why we can replace only one and not both variables of the coproduct) as the Shafarevich complex:
\[
T_S(L) \cofi T_S(L \oplus L[1])  \acfib S \,,
\]
where the differential in $T_S(L\oplus L[1])$ is $d\vartheta_i = t_i$ ($\vartheta_i \in L[1]$ basis). 
\begin{Lem}
\label{eq:tslcofi}
The projection map $T_S(L\oplus L[1]) \to S$ is a quasi-isomorphism.
\end{Lem}
\begin{proof}
For any $s \in S$ we denote the corresponding elements in $L, L[1]$ by $t_s, \vartheta_s$, respectively. We define a super derivation $h: T_S(L\oplus L[1])_\bullet \to T_S(L\oplus L[1])_{\bullet + 1}$ on the generators of the algebra by:
\[
h(s)=0\,, \quad h(t_s) = \vartheta_s\,, \quad h(\vartheta_s ) =0\,.
\]
Then $h$ is an homotopy between the zero map and the `length' map $l$:
\[
dh + hd = l\,,
\]
where on elementary monomial words $w \in T_S(L\oplus L[1])$, $l(w) = \#(w) w$ is the word itself multiplied by its length (in the natural grading give by the tensor algebra, so it is counting only the number of $t$ and $\vartheta$ in the word $w$). The map $l$ is an isomorphism in homological degrees $\geq 1$, and this proves that $\HH_{\geq 1} (T_S(L\oplus L[1])) = 0$.
\end{proof}

Thus a model for the noncommutative zero locus is a sort of generalised Shafarevich complex 
\begin{equation}
\label{eq:Shaf}
A \amalg_{T_S(L)}^\L S \cong A \amalg_{T_S(L)} T_S(L\oplus L[1]) \cong A \amalg_S T_S(L[1])\,,\quad d\vartheta_i = \delta_i\,.
\end{equation}
\begin{Notation} We denote this complex by $\Sh(A):=A \amalg_S T_S(L[1])$.
\end{Notation}

\begin{Remark}
The zero-th homology of the derived zero locus recovers the underived zero locus
\[
\HH_0 \left( \Sh(A) \right) \cong A \amalg_{T_S(L)} S\, . 
\]
If the higher homologies vanish we could say that the couple $(A,\mathcal{J})$ is a generalised \emph{noncommutative complete intersection} (see \cite{BFR,EG} for the case of noncommutative complete intersection of the form $A=T_kV$ a tensor algebra). Suppose that this is true and, in addition, the structure map $S \cofi A$ is a cofibration. In this case the Shafarevich complex is a cofibrant resolution of the underived zero locus $A/\mathcal{J}$ in the category $\DGA_S$ and it can be used to compute the derived representation functor at the level $\L(-)_\nn : \Ho(\DGA_S) \to \Ho(\CDGA_k)$,
\begin{equation}
\label{eq:drep}
\L( A/\mathcal{J})_\nn \cong  \left(A\amalg_S T_S(L[1])\right)_\nn = A_\nn \otimes_k \Sym(\gl_\nn [1])\, .\end{equation}
We direct the interested reader to \cite{BFR,BKR} for details on the derived representation scheme and to \cite{DA} for details on its relationship with the Koszul complex. In particular \cite{DA} contains the proof of~\eqref{eq:drep} for partial preprojective algebras corresponding geometrically to Nakajima quiver varieties.
\end{Remark}

One particular class of examples of Hamiltonian spaces are cotangent bundles: $T^\ast A= T_A \DDer_S(A)$. In~\eqref{eq:distdelta} we defined the distinguished double derivation $\Delta \in \DDer_S(A)$ which is our preferred lifting of the universal derivation $d:A\to \Omega_S^1 A$. The elements obtained by decomposing it as a direct sum of its $I$-graded components~\eqref{eq:deltai}:
\begin{equation}
\Delta_i = e_i \Delta e_i \in T^\ast A\, ,
\end{equation} 
are the natural candidate as gauge elements (with a minus sign):

\begin{Lem}
\label{lem:ncmom}
With the double Poisson structure on $T^\ast A$ given in Proposition~\ref{prop:0shift}, the map
\begin{equation}
\label{eq:D}
\begin{aligned}
 &T_S(L) \to T^\ast A\\
 &t_i \longmapsto \delta_i := - \Delta_i
 \end{aligned}
\end{equation}
is a noncommutative Hamiltonian action.
\end{Lem}
\begin{proof}
We need to show that for any $i \in I$ and any $\omega \in T^\ast A$, we have
\[
\ldb \delta_i, \omega \rdb = \omega e_i \otimes e_i - e_i \otimes e_i \omega\,,
\]
or equivalently, denoting by $\ldb -,-\rdb_\odd$ the double Schouten-Nijunhuis bracket on $T^\ast_\odd A$, that:
\begin{equation}
\label{eq:dpta}
\ldb \Delta_i, \omega \rdb_\odd = \omega e_i \otimes e_i - e_i \otimes e_i \omega\,.
\end{equation}
This can be tested only on the generators of $T^\ast A$, that is either elements $a \in A$ or double derivations $\partial \in \DDer_S(A)$. For elements $a \in A$, equation~\eqref{eq:dpta} is exactly the definition of the double Poisson bracket on $T^\ast_\odd A$. So we only need to prove that
\[
 \ldb \Delta_i,\partial \rdb_{\odd}  = \partial e_i\otimes e_i -e_i \otimes e_i \partial \in (T^\ast A)^{\otimes 2} \,.
\]
This is the content of \cite[Proposition 3.3.1]{VdB}.
\end{proof}


\subsection{Noncommutative Chevalley-Eilenberg and BRST} 
\label{3.4}

In this last Section we define a noncommutative version of the Chevalley-Eilenberg functor and BRST complex needed for the derived Poisson reduction. 

For the moment we start simply from an object in the under category $A \in \DGPPA_{T_S(L)}$ (we do not require the action to be Hamiltonian) and define
\begin{Defn}
The \emph{noncommutative Chevalley-Eilenberg complex} is \begin{equation}
\label{eq:CE}
\CE(A) := A \amalg_{T_S(L)} (T_S(L\oplus L^\ast[-1])) \cong A \amalg_S T_S (L^\ast[-1])\,,
\end{equation}
equipped with the following twisted differential (on the generators $a \in A$ and $\eta_i \in L^\ast[-1]$):
\begin{equation}
\label{eq:dCE}
\begin{cases}
da = d_\old a - [\sum_i \eta_i , a] \,,\\
d \eta_i = -\eta_i^2 =- \frac12 [\eta_i,\eta_i]\,,
\end{cases}
\end{equation}
where $d_\old$ is the original differential on $A$. 
\end{Defn}
In fact, we can write the differential on the Chevalley-Eilenberg complex as the sum of two super-commuting differentials (Lemma~\ref{Lem:d1d2}):
\begin{equation}
\label{eq:d1d2}
d= d_\old + d_\CE\,,
\end{equation}
where $d_\old$ is the old differential on $A$ extended to zero elsewhere and $d_\CE$ is the Chevalley-Eilenberg differential on $T_S(L^\ast[-1])$ ($d_\CE \eta_i =- \eta_i^2$) extended by commutators as indicated in~\eqref{eq:dCE} on $A$.

\begin{Lem}
\label{Lem:d1d2}
The two maps $d_\old$, $d_\CE$ defined in~\eqref{eq:d1d2} are indeed (super-commuting) differentials on $\CE(A)$. This proves that $d=d_\old +d_\CE$ is also a differential.
\end{Lem}
\begin{proof}
It is obvious that $(d_\old)^2=0$ because it is just the old differential, extended to zero elsewhere. For what concerns $d_\CE$, let us call $\eta = \sum_i \eta_i$, and observe that on the generators $d_\CE^2=0$:
\[
\begin{aligned}
&d_\CE^2 (a) = d_\CE\left(-[ \eta ,a] \right) = [\eta^2, a] - [\eta, [\eta,a] ] =\frac12 [[\eta,\eta], a] - [\eta,[\eta,a]] = 0 \,,\\
&d_\CE^2(\eta_i) = d_\CE(-\eta_i^2) =\eta_i \eta_i^2 - \eta_i^2 \eta_i = 0\,.
\end{aligned}
\]
Moreover the two differentials super-commute with each other:
\[
\begin{aligned}
&(d_\old d_\CE + d_\CE d_\old ) (a) = d_\old( -[\eta, a]) - [ \eta, d_\old a] = [\eta,d_\old a] - [\eta,d_\old a] = 0\,,\\
&(d_\old d_\CE + d_\CE d_\old ) (\eta_i) = d_\old (-\eta_i^2) = 0\,.
\end{aligned}
\] 
If follows that also $d=d_\old +d_\CE$ is a differential:
\[
d^2 = d_\old^2 + d_\CE^2 + (d_\old d_\CE + d_\CE d_\old) = 0\,.
\]
\end{proof}

For the moment we could view this Chevalley-Eilenberg construction as a functor:
\begin{equation}
\CE: \DGPPA_{T_S(L)} \to \DGA_{T_S(L\oplus L^\ast [-1])}\,,
\end{equation}
where we momentarily forget the double Poisson structure.

\begin{Remark}
\label{rem:CE}
In the commutative case, the Chevalley-Eilenberg complex $C(\lieg, M)$ can be defined for any $\lieg$-module $M$, without any requirement of a algebra structure on $M$, nor a Poisson structure on $M$ compatible with the $\lieg$ action. However we are interested only in the class of examples of $M=\O(X)$, where $X$ is a Hamiltonian $\lieg$-space, so there is a moment map $X \to \lieg^\ast$ with dual map a Poisson morphism $\Symg \to \O(X)$, so that the action of $x\in \lieg \acts M$ is the Poisson bracket in $M$ with the element $x \in \lieg \subset \Symg \to M$.

Analogously, in the noncommutative case we can define a more general Chevalley-Eilenberg complex for objects $A$ which do not have necessarily a algebra structure, nor a double Poisson structure (essentially only a $T_S(L)$-bimodule structure), but we are only interested in this case for derived Poisson reduction.

\end{Remark}

When we start from a noncommutative Hamiltonian space $A \in \PPA_{T_S(L)}^H$ and we apply the Chevalley-Eilenberg construction to the Shafarevich model for the derived zero locus~\eqref{eq:Shaf} we obtain a noncommutative version of the BRST complex:
\begin{equation}
\CE(\Sh(A)) \cong A \amalg_S T_S(L[1]\oplus L^\ast[-1])\, .
\end{equation}
which now we can equip with a natural double Poisson structure. It is the free product of the double Poisson structure on $A$ together with the obvious one on $T_S(L[1]\oplus L^\ast[-1])$ (with only non-trivial double brackets between generators $\ldb \vartheta_i, \eta_j\rdb = \delta_{ij} e_i\otimes e_i$). 

\begin{Defn}
We define the \emph{noncommutative BRST construction} as the functor:
\begin{equation}
\label{eq:BRST1}
\begin{aligned}
\BRST : &\,\, \PPA_{T_S(L)}^H \to \DGPPA_{T_S(L \oplus L[1]\oplus L^\ast[-1])}\\
&\qquad A \longmapsto A \amalg_S T_S(L[1]\oplus L^\ast[-1])\,.
\end{aligned}
\end{equation}
\end{Defn}

\begin{Remark}
\label{rem:BRSTcharge}
Analogously to the commutative case, the differential on $\BRST(A)$ is induced by the following noncommutative charge (in the sense of Definition~\ref{def:charge})
\begin{equation}
\label{eq:BRSTc}
\gamma=\underbrace{ \sum\limits_i \eta_i \delta_i}_{\gamma_{\mathrm{Ham}}} - \underbrace{\sum\limits_i \eta_i^2\vartheta_i }_{\gamma_{\mathrm{CE}}}\, .
\end{equation}
where the term $\gamma_{\mathrm{Ham}} \in L^\ast[-1]\otimes A$ represents the Hamiltonian action, or Shafarevich differential $L[1] \to A$, and the term $\gamma_\CE \in L^\ast[-1]^{\otimes 2} \otimes L[1]$ represents the multiplication map $S \otimes S \to S$ (once we identify $L\cong S$ and shift the degrees). Indeed one can easily verify that $\{\gamma, \gamma\}$ lies in the commutators subspace and that the differential is obtained by taking the single Poisson bracket with $\gamma$:
\begin{equation}
\label{eq:BRSTdc}
d_\BRST= \{ \gamma, - \} \, .
\end{equation}
\end{Remark}

We conclude this Section with a remark that there are natural maps of dg algebras linking the objects involved in the 2-step derived noncommutative Poisson reduction (from the zero locus to the derived zero locus to the BRST complex):
\begin{equation}
\label{eq:link}
\BRST(A) \xrightarrow{\ev_{\eta=0}}  \Sh(A) \fibr A/\mathcal{I}\,,
\end{equation}

where the map $\ev_{\eta=0}$ is the map sending all the Chevalley-Eilenberg generators $\eta_i \to 0$, and the map $\Sh(A) \fibr A/\mathcal{I}$ is the map sending first the Shafarevich generators $\vartheta_i \to 0$ and then taking the quotient by the two-sided ideal $\mathcal{I}$ generated by their differentials ($d\vartheta_i = \delta_i$). The fact that the map $\ev_{\eta=0}$ is a map of dg algebras follows essentially from the fact that once we set $\eta_i =0$ the Chevalley-Eilenberg differential becomes just the old differential on $\Sh(A)$.


\section{Representation schemes} 
\label{4}

In this section we want to show that when we apply the (opportune version of the) representation functor to the various objects that we defined in Section~\ref{2} and~\ref{3} we obtain the corresponding objects in the commutative world, therefore justifying our definitions (Hamiltonian spaces, derived zero loci, Chevalley-Eilenberg functor, \dots) according to the general Kontsevich-Rosenberg principle (\cite{KR}).
The ground algebra is always $S=kI$, a finite dimensional algebra of orthogonal idempotents. We fix a dimension vector $\nn=(n_i)_{i\in I}$ and consider the representation $\rho_\nn:S \to E_\nn :=\End(k^{n})$, $n=\sum_i n_i$, corresponding to the $S$-bimodule structure
\[
(E_\nn)_{ij}= \Hom_k (k^{n_i},k^{n_j})\,.
\]

\begin{Notation} Because now we need another set of indices running from $1$ to $n$, to avoid confusion we denote by $i,j \in I$, and by $r,s,u,v, \,(\dots) \in \{1,\dots ,n\}$.
\end{Notation}
Given $A\in \DGA_S$, we denote by $\Rep_\nn(A) \in \DGAff_k$ the differential graded affine scheme of $S$-representations of $A$ in $k^n$, and by $A_\nn =\O\Rep_\nn(A) \in \CDGA_k$ its commutative graded algebra of global functions. There is a group scheme action on the representation scheme $\Rep_\nn(A)$ by the group of automorphisms of $k^n$ which preserves the representation $\rho_\nn :S \to E_\nn$, which in this case
\[
G_S:= \{ g \in \Aut(k^n) \,| \, g \rho_\nn(s) = \rho_\nn(s) g \, \forall s \in S\} \cong \prod\limits_{i\in I} \GL_{n_i} =: \GL_\nn\,,
\]
is a product of general linear groups.

We recall that $A_\nn$ is linearly spanned by elements $a_{rs}$, with $a \in A$, and indices $r,s=1,\dots ,n$, with relations:
\begin{equation}
\label{eq:reprel}
(ab)_{rs} = \sum\limits_{u=1}^n a_{ru}b_{us}\,, \quad (\lambda a + \mu b)_{rs} = \lambda a_{rs} + \mu b_{rs}\,, \quad (\lambda,\mu \in k)\,.
\end{equation}
A more abstract way to view $A_\nn$ is $A_\nn \cong (\sqrt[\nn]{A})_{\n\n}$ (see for example \cite{BFR,BKR,DA}), where 
\[
\sqrt[\nn]{A} = \left( E_\nn \ast_S A\right)^{E_\nn} = \{ x \in E_\nn \ast_S A \, | \, x \cdot e = e \cdot x \,\, \forall e \in E_\nn \}
\]
and $(\sqrt[\nn]{A})_{\n\n} =(\sqrt[\nn]{A})/ \langle [\sqrt[\nn]{A},\sqrt[\nn]{A}] \rangle$ is the commutative dg algebra obtained when taking the quotient by the two-sided ideal generated by (graded) commutators.


\subsection{Representation schemes of double Poisson algebras}
\label{4.1}

If $A$ is a double Poisson algebra and we regard $E_\nn$ as a double Poisson algebra with the zero bracket, we have an induced free product double Poisson structure on $E_\nn \ast_S A$, which restricts to a double Poisson structure on $\sqrt[\nn]{A}$ (easy to verify), and descends to a well-defined double Poisson structure on $A_\nn$ (see \cite{VdB}). In terms of the generators $\{a_{rs}\}$:

\begin{Thm}[\cite{VdB}]
\label{thm:dpobjects}
If $A \in \DGPPA_S$, the algebra of functions on its representation scheme $A_\nn$ has a natural induced double Poisson bracket, which is defined on its generators by
\begin{equation}
\label{eq:doublePrep}
\ldb a_{rs}, b_{uv} \rdb = \ldb a,b\rdb'_{us} \otimes \ldb a,b\rdb''_{rv}\,.
\end{equation}
\end{Thm}

\begin{Remark}
\label{rem:charge}
Let us denote the trace map by:
\begin{equation}
\label{eq:trace}
\tr : A_\n= A/[A,A] \to A_\nn^{\GL_\nn}\,.
\end{equation}
In the special double Poisson algebras with differential $d=\{\gamma,-\}$ given by a noncommutative charge, the induced differential on $A_\nn$ is obtained as the Poisson bracket with the trace of the charge: $d=\{ \tr(\gamma),-\}$. This can be easily verified on the generators:
\[
\begin{aligned}
&d (a_{rs}) =(da)_{rs} =(\{\gamma, a\})_{rs} = (\ldb \gamma,a\rdb' \ldb \gamma, a\rdb'')_{rs} = \\
&=\sum\limits_{u} \ldb \gamma,a\rdb_{ru}' \ldb \gamma,a\rdb_{us}'' = \sum\limits_{u} \{\gamma_{uu} , a_{rs} \} = \{ \tr(\gamma) ,a_{rs} \}\, .
\end{aligned}
\]
\end{Remark}

The natural question after reading Theorem~\ref{thm:dpobjects} is whether a morphism of double Poisson algebras induces a Poisson morphism between the induced structures on representation schemes, and the answer is yes:

\begin{Thm}
\label{thm:dpmorphisms}
If $f:A \to  B $ is a morphism of double Poisson algebras, the induced map $f_\nn: A_\nn \to B_\nn$ is a morphism of double Poisson algebras.
\end{Thm}
\begin{proof}
It is straightforward to prove on the generators $a_{rs} \otimes a'_{uv} \in A_\nn^{\otimes 2}$:
\[
\begin{aligned}
&f_\nn^{\otimes 2}\left(  \ldb a_{rs}, a'_{uv} \rdb\right)= f_\nn^{\otimes 2}\left(  \ldb a,a' \rdb'_{us} \otimes \ldb a,a' \rdb''_{rv} \right)= f(\ldb a,a'\rdb')_{us} \otimes f(\ldb a,a'\rdb'')_{rv} = \\
& \ldb f(a),f(a')\rdb')_{us} \otimes \ldb f(a),f(a')\rdb''_{rv} = \ldb f(a)_{rs} , f(a')_{uv} \rdb = \ldb f_\nn(a_{rs}) , f_\nn(a'_{uv})\rdb\,.
\end{aligned}
\]
\end{proof}

Theorem~\ref{thm:dpobjects} and~\ref{thm:dpmorphisms} prove that the representation functor at the level of double Poisson algebras is indeed a functor. Because in the commutative world we are interested only in the single Poisson structure, we can view this `Poisson representation functor' as a functor $(-)_\nn : \DGPPA_S \to \CDGPA_k$, an enrichment of the classical (dg)-representation functor:

\begin{equation}
\label{eq:enrich}
\begin{tikzcd}
  \DGPPA_S  \arrow{r}{} \arrow{d}{}
    & \CDGPA_k \arrow{d}{}  \\
\DGA_S \arrow{r}{} 
&\CDGA_k
 \end{tikzcd}
\end{equation}
(the vertical arrows are the natural forgetful functors).

\begin{Remark}
A natural question that arises in the mind of the reader who knows that the classical representation functor $(-)_\nn : \DGA_S \to \CDGA_k$ has a total left derived functor $\L(-)_\nn : \Ho(\DGA_S) \to \Ho(\CDGA_k)$ (see \cite{BKR,BFR,DA}) is whether or not also the functor at the level of double Poisson algebras can be derived. As pointed out by Y. Berest, it is not known whether such categories of noncommutative Poisson algebras (double Poisson in our case, $\HH_0$-Poisson in the case of \cite{BCER}) have a model structure compatible with the standard (projective) model structure on dg algebras. We can make the following observations
\begin{enumerate}
\item It is possible to define a ``homotopy category'' $\Ho^\ast(\DGPPA_S)$ with objects the double Poisson algebras which are also cofibrant as dg algebras over $S$, and morphisms the homotopy classes of morphisms (where the notion of homotopy is a double Poisson version of polynomial homotopy of dg algebras, as explained in \cite[Remark 5.1.1]{BCER}). With these definitions the representation functor at the level of double Poisson algebras preserves polynomial homotopies, and therefore admits a total left derived functor 
\[
\L(-)_\nn: \Ho^\ast(\DGPPA_S) \to \Ho(\CDGA_k)\,,
\]
which is computed as $\L(A)_\nn = (QA)_\nn$, where $S \cofi QA \acfib A$ is a cofibrant replacement in the category of dg algebras over $S$ and the maps are all maps of double Poisson algebras.
\item Analogously to what suggested in \cite{BCER} for $\HH_0$-Poisson algebras one could consider first the infinite-dimensional limit (to eliminate the dependence on the dimension vector $\nn$):
\[
(-)_{\bm\infty}: \DGA_S \to \CDGA_k\,,
\]
which has a total left derived functor $\L(-)_{\bm \infty}$, and use the construction of \emph{homotopy pull-back of model categories along functors with total left derived functors} (see \cite{To1}) --- in this case along the forgetful functor $\CDGPA_k \to \CDGA_k$ --- to obtain a model category $\DGA_S \times_{\CDGA_k}^h \CDGPA_k$, which because of the diagram~\eqref{eq:enrich} comes with a functor from double Poisson algebras:
\[
\DGPPA_S \to \DGA_S \times_{\CDGA_k}^h \CDGPA_k \,,
\]
which is homotopy invariant and conjecturally it induces an equivalence of categories: $\Ho^\ast(\DGPPA_S) \to \Ho(\DGA_k \times_{\CDGA_k}^h \CDGPA_k)$. We do not know whether this is true or not, and it could be part of future work.
\end{enumerate}
\end{Remark}


\subsection{Hamiltonian spaces}
\label{4.2}

In this Section we show how we obtain classical Hamiltonian $\gl_\nn$-spaces by applying the representation functor to noncommutative Hamiltonian spaces defined in \S\ref{3.3}. Let us start by considering the algebra $T_S(L)$ from~\eqref{eq:L}, a free product over $S$ of copies of polynomials in $1$ variable, with double Poisson structure:
\begin{equation}
\label{eq:dpkt}
\ldb t_i,t_j\rdb = \delta_{ij} (  t_i \otimes e_i -e_i \otimes t_i)\,.
\end{equation}

For some dimension vector $\nn$, the representation scheme of $T_S(L)$ is just a product of general linear algebras, in fact:
\begin{equation}
\label{eq:gln}
\Rep_\nn( T_S(L)) \cong \prod_{i \in I} \Rep_\nn (k[t_i]) \cong \prod_{i\in I} \Hom_{S-\Bimod}(k\cdot t_i,\gl_n) \cong \prod_{i \in I} \gl_{n_i} (k) =: \gl_\nn\,.
\end{equation}
We identify its algebra of functions with $\O(\gl_\nn)=\Sym(\gl_\nn^\ast) \cong \Sym(\gl_\nn)$ using the trace and we can show that the Poisson bracket induced from the double Poisson bracket on $T_S(L)$ is the natural Poisson structure on $\Sym(\gl_\nn)$:

\begin{Prop}
The induced Poisson structure on $\Sym(\gl_\nn) \cong T_S(L)_\nn$ is the natural extension of the Lie bracket.
\end{Prop}
\begin{proof}
Because of the decomposition in~\eqref{eq:gln} we can just prove the case with one vertex: $|I|=1$. This case is explained in \cite[Example 7.5.3.]{VdB}.
\end{proof}

From now on we will make implicit use of the trace isomorphism and therefore identify $T_S(L)_\nn \cong \Sym(\gl_\nn)$ without saying it explicitly again. Other examples of interest later are when we shift the $S$-bimodule $L$ to either $L[1]$ or $L[-1]$. In this case we obtain:
\begin{equation}
\label{eq:tsl1}
T_S(L[1])_\nn =\O(\gl_\nn[-1]) \cong \Sym(\gl_\nn[1])\,, \quad T_S(L[-1])_\nn =\O(\gl_\nn[1]) \cong \Sym(\gl_\nn[-1])\,,
\end{equation}
where $\Sym(-): \DGVect_k \to \CDGA_k$ is the graded commutative one, so it gives the symmetric algebra for even degrees and the antisymmetric algebra for odd degrees $\Sym(V) = S(V_{\mathrm{even}}) \otimes \Lambda(V_{\mathrm{odd}})$. 

Now let us consider a noncommutative Hamiltonian space $A \in \PPA_{T_S(L)}^H$, whose structure map $T_S(L) \to A$ induces a $\GL_\nn$-equivariant map of Poisson algebras $\Sym(\gl_\nn) \to A_\nn$. Because of the property~\eqref{eq:hamilton} this corresponds to the infinitesimal action of $\gl_\nn$ coming from the natural action of $\GL_\nn$ by conjugation --- this is essentially \cite[Proposition 7.11.1]{VdB} --- so that $A_\nn$ is a Hamiltonian $\gl_\nn$-space. Dually, the map of schemes $\mu_\nn: \Rep_\nn(A) \to \gl_\nn^\ast$ is a Poisson moment map for the canonical action of $\GL_\nn\acts \Rep_\nn(A)$. In other words, we obtain the following enriched version of the diagram~\eqref{eq:enrich}:

\begin{Thm}
\label{thm:repham}
The representation functor enriches to a functor between noncommutative Hamiltonian spaces and commutative Hamiltonian spaces:
\begin{equation}
\label{eq:remham}
\begin{tikzcd}
  \PPA_{T_S(L)}^H  \arrow{r}{(-)_\nn} \arrow{d}{}
    & \CPA_{\Sym(\gl_\nn)}^H \arrow{d}{}  \\
\DGA_{T_S(L)} \arrow{r}{(-)_\nn} 
&\CDGA_{\Sym(\gl_\nn)}
 \end{tikzcd}
\end{equation}
where the vertical arrows are the natural forgetful functors. Moreover, the representation functor $(-)_\nn:\DGA_{T_S(L)} \to \CDGA_{\Sym(\gl_\nn)}$ is cocontinuous, so in particular it preserves coproducts.
\end{Thm}
\begin{proof}
The previous considerations prove the upper part of the diagram. As for the cocontinuity, it is a simple consequence of the fact that it is the ``under category version'' of the cocontinuous functor $(-)_\nn : \DGA_S \to \CDGA_k$ and, the fact that colimits in the under category are computed in the original category.
\end{proof}

If we denote by $\mathcal{I}$ the two-sided ideal in $A$ of the image of the Hamiltonian action $T_S(L)\to A$, it follows from the property of cocontinuity of the representation functor (in particular, coproducts) that the zero locus $A/\mathcal{I}=A \amalg_{T_S(L)} S$ corresponds to the classical zero locus of the induced moment map $\mu_\nn: \Rep_\nn(A) \to \gl_\nn^\ast$:
\begin{equation}
\begin{tikzcd}
  T_S(L) \arrow{r}{} \arrow{d}{}
    & S \arrow{d}{}  \\
A \arrow{r}{} 
& A\amalg_{T_S(L)} S 
\end{tikzcd}
\quad \overset{(-)_\nn}{\rightarrow} \quad 
\begin{tikzcd}
 \Sym(\gl_\nn) \arrow{r}{} \arrow{d}{}
    & k \arrow{d}{}  \\
A_\nn \arrow{r}{} 
& A_\nn \otimes_{\Sym(\gl_\nn)} k =\O(\mu_\nn^{-1}(0))
\end{tikzcd}
\end{equation}

\begin{Remark}[Poisson structure on the zero locus] As for the Poisson structure we observe that, by Remark~\ref{rem:PS}, $A\amalg_{T_S(L)} S $ has an induced $\HH_0$-Poisson structure, which induces (Remark~\ref{rem:H0CB}) a Poisson structure on the ordinary Poisson reduction
\[
\Rep_\nn(A/\mathcal{I})\sslash \GL_\nn = \Spec( \O(\mu_\nn^{-1}(0))^{\GL_\nn})\, .
\] 

\end{Remark}

Finally when we consider the specific chosen model representing the derived zero locus as the Shafarevich complex $\Sh(A) = A \amalg_{T_S(L)} T_S(L \oplus L[1])$ we obtain:

\begin{Lem}
\label{lem:shakos}
The Shafarevich complex (a chosen model for the noncommutative derived zero locus) is sent under the representation functor to the standard model for the commutative derived zero locus, the Koszul complex:
\begin{equation}
\label{eq:shakos}
\left( \Sh(A)\right)_\nn \cong A_\nn \otimes_{\Sym(\gl_\nn)} \Sym(\gl_\nn \oplus \gl_\nn[1])= A_\nn \otimes_k \Sym(\gl_\nn[1])\,.
\end{equation}
\end{Lem}

\begin{Notation}
For a noncommutative Hamiltonian space $A$, we denote the associated commutative Koszul complex of dimension $\nn$ by
\begin{equation}
\C_\nn(A):= (\Sh(A) )_\nn \cong A_\nn \otimes \Sym(\gl_\nn[1])\,.
\end{equation}
\end{Notation}
Last but not least --- because it represents the largest class of examples of Hamiltonian spaces --- we show the correspondence between noncommutative and commutative cotangent bundles:
\begin{Thm}
If $A$ is smooth, there is a natural isomorphism between the representation scheme of the noncommutative cotangent bundle of $A$ and the ordinary cotangent bundle of the representation scheme of $A$:
\begin{equation}
\label{eq:TastRep}
\Rep_\nn( T^\ast A ) \cong T^\ast \Rep_\nn(A) \,.
\end{equation}
The Poisson structure induced on $T^\ast \Rep_\nn(A)$ by the double Poisson structure on $T^\ast A$ is the Poisson structure coming from its standard symplectic structure.
\end{Thm}
\begin{proof} Let us consider the Van den Bergh's functor $(-)_\nn : A-\Bimod \to A_\nn-\Mod$, which has the property that for $\D \in A-\Bimod$ and any positive shifting $r=0,1,\dots$:
\[
(T_A \D[r])_\nn \cong \Sym_{A_\nn} (\D_\nn[r])\,.
\]
When algebra $A$ is smooth, the Van den Bergh's functor sends the $A$-bimodule of double derivations to the $A_\nn$-module of its derivations: $\DDer_S(A) \to \Der(A_\nn)$ (\cite[Proposition 3.3.4]{VdB1}). As a consequence the representation functor sends odd and even cotangent bundle to the odd and even cotangent bundles of the representation scheme:
\[
\begin{aligned}
&(T^\ast_\odd A )_\nn \cong \Sym_{A_\nn}(\Der(A_\nn)[1]) = \O ( T^\ast[-1] \Rep_\nn(A))\, ,\\
& (T^\ast A)_\nn \cong \Sym_{A_\nn} (\Der(A_\nn)) = \O(T^\ast \Rep_\nn(A))\, .
\end{aligned}
\]
\cite[Proposition 7.6.1]{VdB} proves that the induced bracket on $T^\ast[-1]\Rep_\nn(A)$ is the Schouten-Nijenhuis bracket. But then, because of the relationship between the natural Poisson structure on $T^\ast\Rep_\nn(A)$ and the Schouten-Nijenhuis one (it is --- with a minus sign\footnote{This depends also on the conventions on what is the `natural symplectic form on a cotangent bundle'. Here we take the one for which the Poisson bracket $\{x,y\} =1$ if $x$ denotes a coordinate on the basis and $y$ on the fiber.} --- the same on the generators), and the relationship between the double Poisson structures on $T^\ast_\odd A$ and $T^\ast A$ (Theorem~\ref{thm:SN} and Proposition~\ref{prop:0shift}) the result follows.
\end{proof}


\subsection{Chevalley-Eilenberg and BRST}
\label{4.3}

In this Section we complete the last step of the two-step derived Poisson reduction by considering the Chevalley-Eilenberg functor and its composite with the derived zero locus, the BRST construction. We show that the noncommutative definitions in \S\ref{3.4} give the corresponding one in the commutative world.

Let us start from the Chevalley-Eilenberg functor~\eqref{eq:CE} which is the coproduct over $T_S(L)$ with $T_S(L\oplus L^\ast[-1])$, with a twisted differential. By the discussion in~\eqref{eq:tsl1} this differential graded algebra ($d\eta_i=- \eta_i^2$) is sent under the representation functor to the commutative Chevalley-Eilenberg complex for the module $\Sym(\gl_\nn)$:
\begin{equation}
T_S(L\oplus L^\ast[-1])_\nn \cong \Sym(\gl_\nn) \otimes_k \Sym(\gl_\nn^\ast[-1]) = C(\gl_\nn, \Sym(\gl_\nn))\, ,
\end{equation}
As a consequence (using the monoidal properties of the representation functor) it is easy to see that for any noncommutative Hamiltonian space $A \in \DGPPA_{T_S(L)}$ the Chevalley-Eilenberg complex corresponds to the usual one:
\begin{equation}
\label{eq:CEA}
\CE(A)_\nn \cong A_\nn \otimes_k \Sym(\gl_\nn^\ast[-1]) = C(\gl_\nn, A_\nn)\,.
\end{equation} 
In fact the twisted differential on $C(\gl_\nn,A_\nn)$ is obtained by adding the terms commutators of Chevalley-Eilenberg generators with elements of $A_\nn$, and it is induced exactly from the twisting defined on the noncommutative Chevalley-Eilenberg complex~\eqref{eq:dCE}.

In other words we have the following result:
\begin{Lem}
\label{thm:cerep}
There is a commutative diagram
\begin{equation}
\label{eq:cerep}
\begin{tikzcd}
  \DGPPA_{T_S(L)}  \arrow{r}{(-)_\nn} \arrow{d}{\CE}
    & \CDGPA_{\Sym(\gl_\nn)} \arrow{d}{C(\gl_\nn, -)}  \\
\DGA_{T_S(L\oplus L^\ast[-1])} \arrow{r}{(-)_\nn} 
&\CDGA_{C(\gl_\nn,\Sym(\gl_\nn))}
 \end{tikzcd}
\end{equation}
\end{Lem}

As a consequence also the noncommutative BRST construction in~\eqref{eq:BRST1} becomes the commutative one:
\begin{equation}
\label{eq:BRST2}
(\BRST(A))_\nn = (A \amalg_S T_S(L[1]\oplus L^\ast[-1]) )_\nn \cong A_\nn \otimes_k \Sym( \gl_\nn[1] \oplus \gl_\nn^\ast[-1])\, .
\end{equation}

The double Poisson structure on $\BRST(A)$ (free product of the one on $A$ and the canonical one on $T_S(L[1]\oplus L^\ast[-1])$), induces the Poisson structure on~\eqref{eq:BRST2} which is the tensor product of the one on $A_\nn$ and the canonical Poisson structure on $\Sym(\gl_\nn[1]\oplus\gl_\nn^\ast[-1])$ (the extension of the natural pairing $\gl_\nn \otimes \gl_\nn^\ast \to k$).

\begin{Thm}
\label{thm:BRST}
There is a commutative diagram between the noncommutative and the commutative BRST constructions:
\begin{equation}
\label{eq:BRSTrep}
\begin{tikzcd}
  \PPA_{T_S(L)}^H  \arrow{r}{(-)_\nn} \arrow{d}{\BRST}
    & \CPA_{\Sym(\gl_\nn)}^H \arrow{d}{\BRST}  \\
\DGPPA_{T_S(L\oplus L[1]\oplus  L^\ast[-1])} \arrow{r}{(-)_\nn} 
&\CDGPA_{\Sym(\gl_\nn \oplus \gl_\nn[1] \oplus \gl_\nn^\ast[-1])}
 \end{tikzcd}
\end{equation}
\end{Thm}

\begin{Notation} For a noncommutative Hamiltonian space $A$, we denote the associated commutative BRST complex of dimension $\nn$ (the object obtained in the right-bottom side of the diagram~\eqref{eq:BRSTrep} by following either of the two paths): 
\begin{equation}
\label{eq:BRST3}
\B_\nn(A):=\BRST(A)_\nn \cong A_\nn \otimes \Sym(\gl_\nn[1]\oplus\gl_\nn^\ast[-1])\,.
\end{equation}
\end{Notation}

\begin{Remark} 
Remember that the differential on the noncommutative BRST complex is induced by a charge $d_\BRST =\{ \gamma, -\}$ (Remark~\ref{rem:BRSTcharge}). As a consequence the induced differential on the commutative BRST complex $\B_\nn(A)$ is also induced by a charge, obtained as the trace of the noncommutative charge (see Remark~\ref{rem:charge}):
\[
d_{\B_\nn(A)} = \{ \tr(\gamma) , - \}\,.
\]
It is easy to verify that this is indeed the usual BRST charge associated to the $\gl_\nn$ action on $A_\nn$.
\end{Remark}

Finally, we notice that from the noncommutative dg algebra maps~\eqref{eq:link} linking the BRST complex, the Shafarevich complex and the zero locus, we obtain the analogous well known (commutative) dg algebra maps:
\begin{equation}
\label{eq:link2} 
\B_\nn(A) \xrightarrow{ {(\ev_{\eta=0})}_\nn} \C_\nn(A) \fibr \O(\mu_\nn^{-1}(0))\, . 
\end{equation}

\section{Some homological computations and examples}
\label{5}

We devote this last Section to some homological computations (commutative BRST homology) and some examples, such as noncommutative cotangent bundles of path algebras of quivers, and in particular the example of the Jordan quiver --- the scheme of commuting matrices --- and similar associated schemes.


\subsection{Computation of the Chevalley-Eilenberg (co)homology}
\label{5.1}

In this Section we compute the (commutative) Chevalley-Eilenberg homology for representation algebras. Let us recall first that in our conventions (differentials of degree $-1$) the Chevalley-Eilenberg complex for a $\lieg$-module $M$ is the following chain complex concentrated in non-positive degrees:
\begin{equation}
C(\lieg, M) := \Hom_k(\Sym(\lieg[1]),M) = [ M \longrightarrow \lieg^\ast \otimes M \longrightarrow \Lambda^2 \lieg^\ast \otimes M \longrightarrow  \dots ]\, .
\end{equation}
Its homology is essentially the Lie algebra cohomology (the cohomology of the usual cochain complex version of Chevalley-Eilenberg):
\[
 \HH_\bullet( C(\lieg, M )) =\HH^{-\bullet}(\lieg,M)\, .
\]

Let us start from a double Poisson algebra $A \in \DGPPA_{T_S(L)}$ and, for a dimension vector $\nn$, the induced Poisson morphism $\Sym(\gl_\nn) \to A_\nn$ and the associated Chevalley-Eilenberg complex:
\[
C(\gl_\nn, A_\nn)\,.
\]
The calculation of its homology is a somewhat classical result:

\begin{Thm}
\label{thm:CEhom}
Let $A\in \DGPPA_{T_S(L)}$ and $\nn$ a dimension vector. The Chevalley-Eilenberg homology of the $\gl_\nn$-module $A_\nn$ is 
\begin{equation}
\label{eq:CEhom}
\HH_\bullet( C(\gl_\nn, A_\nn)) \cong \HH_\bullet(A_\nn)^{\GL_\nn} \otimes \HH^{-\bullet}(\gl_\nn,k)\, .
\end{equation} 
\end{Thm}
\begin{proof}
We recall that the Chevalley-Eilenberg complex for $A_\nn$ is obtained applying the representation functor to the noncommutative one: $C(\gl_\nn,A_\nn) \cong (\CE(A))_\nn$. The differential on $C(\gl_\nn,A_\nn)$ is the sum of the two super-commuting differentials (induced from the two super-commuting differentials~\eqref{eq:d1d2}):
\[
d= d_1 + d_2\,,
\]
where $d_1$ is the old differential on $A_\nn$ extended to zero elsewhere and $d_2$ is the Chevalley-Eilenberg differential extended to commutators on $A_\nn$ (the $\gl_\nn$ action on $A_\nn$ induced from $\GL_\nn \acts A_\nn$). To compute the homology we can use the `classical' technique consisting of the following steps:
\begin{enumerate}
\item We consider the filtration 
\[
C(\gl_\nn,A_\nn)= F_0 \supset F_1 \supset F_2 \supset \dots \supset F_{N} \supset F_{N+1} = 0 \,, \qquad \left(N= \dim(\gl_\nn) = \sum_i n_i^2\right)\,,
\]
where $F_p$ is the linear span of monomials containing at least $p$ terms $\eta$, and consider the associated spectral sequence $\{E^r_{p,q}\}$.
\item The differential on the associated graded (zero page) contains only $d_1$, therefore:
\[
E^1 \cong \Sym(\gl_\nn^\ast[-1])\otimes \HH_\bullet(A_\nn) = C\big(\gl_\nn ,\HH_\bullet(A_\nn)\big)\,.
\]
with $d^{(1)}: E^1 \to E^1$ being the Chevalley-Eilenberg differential for the $\gl_\nn$ graded module $\HH_\bullet(A_\nn)$ (now without differential).
\item Because $\gl_\nn$ is a finite dimensional reductive Lie algebra, the Chevalley-Eilenberg homology of the graded module $\HH_\bullet(A_\nn)$ is 
\[
E^2 \cong \HH_\bullet(A_\nn)^{\GL_\nn} \otimes \HH^{-\bullet}(\gl_\nn,k) \,.
\]
\item By (bi)degree inspection the differential on the second page is zero $d^{(2)}=0$. Therefore the sequence collapses at $E^2=E^3=\dots =  E^\infty$ and the second page computes the homology.
\end{enumerate}
\end{proof}

Now let us start from a noncommutative Hamiltonian space $A \in \mathtt{PPAlg}_{T_S(L)}^H $, and consider for some dimension $\nn$ the associated commutative BRST complex, which is the Chevalley-Eilenberg complex for the Koszul complex:
\begin{equation}
\label{eq:BRST3}
\BRST(A_\nn) =\B_\nn(A) = C(\gl_\nn, \C_\nn(A) )\, .
\end{equation}
As a corollary of the previous Theorem (using $\Sh(A)$ instead of $A$):
\begin{Cor}
Let $A \in \mathtt{PPAlg}_{T_S(L)}^H$. The homology of the BRST complex for $A_\nn$ is the tensor product of the $\GL_\nn$-invariant homology of the Koszul complex with the Lie algebra (co)homology of $\gl_\nn$:
\begin{equation}
\label{eq:BRST4}
\HH_\bullet(\B_\nn(A)) \cong \HH_\bullet (\C_\nn(A))^{\GL_\nn} \otimes \HH^{-\bullet}(\gl_\nn,k)\,.
\end{equation}
\end{Cor}

The goal would be to actually compute $\HH_\bullet(\C_\nn(A))^{\GL_\nn}$ which in fact, even in the easier cases is not known. For example when $S=k$ (one vertex), and $A=T^\ast k[x] \cong k\langle x,y\rangle$, the zero locus $A/\mathcal{I} =k \langle x,y\rangle /\langle [x,y] \rangle = k[x,y]$ is a polynomial algebra in two variables, and the homology $\HH_\bullet(\C_n(A))^{\GL_n}$ is the $\GL_n$-invariant part of the Koszul homology for the scheme of commuting matrices. There is a conjecture (\cite[Conjecture 1]{BFPRW}) saying that the following diagonal restriction map is an isomorphism:
\[
\HH_\bullet(\C_n(A))^{\GL_n} =\HH_\bullet \left(k [x_{uv},y_{uv},\vartheta_{uv}] \right)^{\GL_n} \to k [x_u,y_u,\vartheta_u]^{S_n}\,.
\]

We are not able to fully compute the homology $\HH_\bullet(\C_n(A))^{\GL_n}$ but, using the Poisson structure on $\B_n(A)$ and trace maps, we can give another proof to a decomposition of $\HH_\bullet(\C_n(A))^{\GL_n}$ which was not previously known to us. We postpone the result to \S\ref{5.4}, after we have introduced the class of examples which are cotangent bundles of quiver path algebras.


\subsection{Path algebras of quivers}
\label{5.2}

In this Section we work out derived noncommutative Poisson reduction for cotangent bundles of path algebras of quivers $T^\ast kQ$. The reader who wants to have a more detailed introduction on the noncommutative geometry of quiver path algebras can read the main references \cite{CEG},\cite{EG}.

We fix $S=kI$ and we consider the path algebra $kQ$ of a quiver $Q=(Q_0=I,Q_1)$ with vertex set $I$. Path algebras are the simplest examples of algebras over $S$, because they are in fact free algebras over $S$, as we can write 
\[
kQ = T_S(M)\,,
\]
where $M=kQ_1$ is the $S$-bimodule which, as a vector space, is freely generated by the arrows. The space of double derivations for $kQ$ is rather easy to compute. For each arrow of the quiver $x \in Q_1$ we consider the following double derivation $\partial_x \in \DDer_S(kQ)$ defined on generators by
\begin{equation}
\label{eq:dex}
\partial_x(y) = \delta_{x,y} e_{t(x)} \otimes e_{s(x)} \,,
\end{equation}
where $s,t:Q_1 \to I$ are, respectively, the source and target map (we use the convention of concatenation of paths \emph{from right to left}). It is easy to check that $\DDer_S(kQ)$ is the free $kQ$-bimodule with generators $\{\partial_x\}_{x\in Q_1}$. We recall that the $kQ$-bimodule structure on $\DDer_S(kQ)$ is induced by the inner bimodule structure on $(kQ)^{\otimes 2}$, explicitely:
\[
(a \cdot \partial \cdot b )(c) = a \ast (\partial(c)) \ast b , \quad a,b,c\in A, \partial \in \DDer_S(kQ)\,.
\]
Because $S\subset kQ$ we can view $\DDer_S(kQ)$ also as a $S$-bimodule, that is a $I\times I$-graded vector space, and it is easy to see that for each arrow $x \in Q_1$ going from $i=s(x)$ to $j=t(x)$, the double derivation $\partial_x \in  \DDer_S(kQ)_{ji} $ is an ``arrow'' going in the opposite direction, from $j$ to $i$, in fact:
\begin{equation}
\label{eq:dualarrow}
(e_k \cdot \partial_x \cdot e_l )(y) = e_k \ast ( \delta_{x,y} e_j\otimes e_i) \ast e_l= \delta_{x,y} e_j e_l \otimes e_k e_i = \delta_{jl}\delta_{ik} \partial_x(y)\,, \quad \forall y \in Q_1\,.
\end{equation}
If we consider the doubled quiver $\overline{Q}$, obtained from $Q$ by adding for each arrow $x \in Q_1$ a dual arrow $x^\ast$ going in the opposite direction, we have:
\begin{Prop}
\label{prop:dualquiver}
There is a natural isomorphism between the noncommutative cotangent bundle of $kQ$, and the path algebra of the double quiver $k\overline{Q}$, given by:
\begin{equation}
\label{eq:dualquiver}
\begin{aligned}
 &k\overline{Q} \xrightarrow{\sim} T^\ast \left( kQ\right)=T_{kQ}(\DDer_S(kQ))\\
 &\begin{cases}
 x \longmapsto x\,,\\
 x^\ast \longmapsto \partial_x\,.
 \end{cases}
\end{aligned} 
\end{equation}
\end{Prop}
We remark that the above isomorphism is an isomorphism of $kQ$-bimodules, and in particular of $S$-bimodules, that is $I\times I$-graded vector spaces (as shown in~\eqref{eq:dualarrow}). In terms of the $kQ$-bimodule basis $\{ \partial_x \}_{x\in Q_1} \subset  \DDer_S(kQ)$, the distinguished derivation is
\begin{equation}
\label{eq:momentelement}
\Delta = \sum\limits_{x \in Q_1}  [\partial_x,x ]\,.
\end{equation}
In fact it is sufficient to prove the equality~\eqref{eq:momentelement} on each arrow $y \in Q_1$ of the quiver:
\[
\begin{aligned}
&\sum\limits_{x\in Q_1} [\partial_x,x ] (y) = \sum\limits_{x\in Q_1}  \partial_x(y)\ast x -x\ast \partial_x(y) = (e_{t(y)} \otimes e_{s(y)})\ast y - y \ast ( e_{t(y)}\otimes e_{s(y)})= \\
&=y\otimes e_{s(y)} - e_{t(y)} \otimes y   =\Delta(y)\,.
\end{aligned}
\]

By the general construction explained in \S\ref{3.2}, the cotangent bundle $T^\ast (kQ)$ carries a natural double Poisson structure.
\begin{Prop}
\label{prop:dPdualquiver}
Under the isomorphism~\eqref{eq:dualquiver} the induced double Poisson structure on the path algebra of the doubled quiver $k\overline{Q}$ with only non-zero brackets the one pairing an arrow $x$ and its dual $x^\ast$ by:
\begin{equation}
\label{eq:dPdualquiver}
\ldb x , x^\ast \rdb=  e_{s(x)} \otimes e_{t(x)}\,, \quad \left( \ldb x^\ast ,x \rdb = - e_{t(x)}\otimes e_{s(x)}\right) \,.
\end{equation}
The noncommutative moment map $T_S(L) \to k\overline{Q}$ is the one sending $ t_i \mapsto \delta_i = e_i \left(\sum_x [x ,x^\ast] \right) e_i$.
\end{Prop}
\begin{proof}
Given $x,y\in Q_1$ we have to show that the double bracket $\ldb \partial_x,\partial_y\rdb $ given in~\eqref{eq:dPcot} is zero. This follows from the fact that any elementary double derivation $\partial_x$ sends any arrow $z\in Q_1$ into $S\otimes S$, therefore the triple derivation
\[
\ldb \partial_x,\partial_y \rdb_l^\sim = (\partial_x \otimes 1)\circ \partial_y - (1\otimes \partial_y)\circ \partial_x
\] 
vanishes on the arrows, and so it vanishes everywhere.
\end{proof}

\begin{Remark} The double Poisson on $k\overline{Q}$ induces a $\HH_0$-Poisson structure on $(k\overline{Q})_\n$ (the space of "cyclic paths"), which is the well-known Necklace Lie algebra structure (\cite{BLB}).
\end{Remark}

The noncommutative zero locus is:
\begin{equation}
\label{eq:redpath}
T^\ast (kQ) / \langle \delta \rangle \cong k\overline{Q} / \bigg\langle \sum_{x \in Q_1} [x,x^\ast] \bigg\rangle =:\Pi(Q)\, ,
\end{equation}
the so called \textit{preprojective algebra} of the quiver $Q$. The Shafarevich complex is 
\begin{equation}
\label{eq:shaqui}
\Sh(kQ) \cong k\overline{Q} \amalg_S T_S(L[1]) \cong k{Q^\vartheta}\,,
\end{equation}
where we denote by $Q^\vartheta$ the quiver obtained from the doubled quiver $\overline{Q}$ by adding a new loop ($\vartheta_i$) in homological degree $1$ on each vertex, and $kQ^\vartheta$ is its (graded) path algebra. The BRST complex becomes
\begin{equation}
\BRST( kQ) \cong k\overline{Q} \ast_S T_S(L[1]\oplus L^\ast[-1]) \cong k\widehat{Q}\,,
\end{equation}
the path algebra of $\widehat{Q}$, which is the quiver obtained from the doubled quiver $\overline{Q}$ adding two new loops on each vertex, one in homological degree $1$ ($\vartheta_i$) and one in degree $-1$ ($\eta_i$). The BRST differential is, on the generators:
\begin{equation}
\label{eq:BRSTpath}
\begin{cases}
d x =- \left[\eta,x\right]\,,\qquad (\eta=\sum_i\eta_i)\\
d x^\ast =- \left[ \eta, x^\ast \right]\,, \\
d \vartheta_i = \delta_i - \left[\eta_i ,\vartheta_i \right]\,,\\
d \eta_i= -\eta_i^2=-\frac12[\eta_i,\eta_i]\,.
\end{cases}
\end{equation}


\subsection{The scheme of commuting matrices and similar}
\label{5.3}

Let us consider one particular class of examples of path algebras where the quiver has only one vertex. In this case we can only have arrows that are loops on this vertex, so the quiver is uniquely determined by the number of loops, and we denote by $Q_g$ the quiver with $g$ loops. We denote by $x_1,\dots x_g$ its loops and by $y_1, \dots ,y_g$ their dual loops, so that its cotangent bundle is a free algebra on $2g$ generators.
\[
k\overline{Q_g} = k \langle x_1,y_1,\dots ,x_g,y_g\rangle \,.
\]
The image of $t$ under the moment map $k[t] \to k\overline{Q_g}$ is $\delta=\sum_a [x_a,y_a]$. Underived commutative Poisson reduction gives the affine $GL_n$-quotient of the zero locus of the moment map $\mu_n : (\gl_n)^{\times 2g} \to \gl_n^{(\ast)}$: 
\begin{equation}
\label{eq:riem}
\mu_n^{-1}(0)\sslash \GL_n= \Rep_n\left( k\overline{Q_g}/\mathcal{I} \right)\sslash \GL_n = \left\{ (X_a,Y_a) \in (\gl_n)^{\times 2g} \, | \, \sum_a [X_a,Y_a] =0 \right\}  \sslash \GL_n \, ,
\end{equation}
a linearised (Lie algebra) version of the $\GL_n$-character variety of a Riemann surface of genus $g$: $\Hom_\Grp(\pi_1 (\Sigma_g ),\GL_n)\sslash{\GL_n}$. The noncommutative Shafarevich complex and the BRST complex are, respectively:
\[
\begin{aligned}
&\Sh( k\overline{Q_g}) = k \langle x_1,y_1,\dots ,x_g,y_g, \vartheta \rangle ,,\qquad d\vartheta = \sum_a [x_a,y_a]\,,\\
& \BRST(k\overline{Q_g}) = k\langle x_1,  y_1,\dots ,x_g,y_g, \vartheta,\eta\rangle\,,\qquad d = d_{\mathrm{Sh}} + d_\CE \,.
\end{aligned}
\]
For $g\geq 2$ the scheme~\eqref{eq:riem} is a complete intersection in $(\gl_n)^{\times 2g}$, so that the homology of the Koszul complex is only in degree zero, and the homology of the commutative BRST complex is essentially just the ring of functions on~\eqref{eq:riem}:
\[
\HH_\bullet(\B_n( k \overline{Q_g} ) ) \cong \HH_\bullet(\C_n(k \overline{Q_g} ) )^{\GL_n}\otimes \HH^{-\bullet } (\gl_n,k) = \O(\mu^{-1}(0))^{\GL_n} \otimes \HH^{-\bullet } (\gl_n,k)\,,
\]
(considering that $\HH^{-\bullet}(\gl_n,k) = k [\tr(\eta),\tr(\eta^3),\dots, \tr(\eta^{2n-1})]$ is an exterior algebra, so essentially `invertible'). 

The case $g=1$ instead corresponds to the scheme of commuting matrices, for which the homology of the Koszul complex in concentrated in degrees $0,1,\dots, n$, and it is not known. There is a conjecture (\cite[Conjecture 1]{BFPRW}), that the diagonal restriction to multisymmetric polynomials is a quasi-isomorphism:
\begin{equation}
\label{eq:conjecture1}
(\C_n(kQ_1) )^{\GL_n} =k[x_{ij},y_{ij},\vartheta_{ij}]^{\GL_n} \to k[x_i,y_i,\vartheta_i]^{S_n}\, .
\end{equation}
Using~\eqref{eq:BRST4} we can rewrite this conjecture by saying that the following map is an isomorphism
\begin{equation}
\HH_\bullet(\B_n(kQ_1)) = \HH_\bullet (k [x_{ij},y_{ij},\vartheta_{ij},\eta_{ij}]) \to  k[x_i,y_i,\vartheta_i]^{S_n}[\tr(\eta),\dots ,\tr(\eta^{2n-1})]\,.
\end{equation}

We conclude this Section with two more examples similar to the commuting scheme: 
\begin{enumerate}
\item
First we consider the cotangent bundle of the ring of Laurent polynomials: $A= T^\ast B=T^\ast(k[x^{\pm1}])$. Any double derivation $B \to B\otimes B$ is uniquely determined by its value on $x$, and by using the $B$-bimodule structure we can obtain any value starting from the following derivation $\partial_x(x) = 1\otimes 1$. $\DDer(B)$ is a free $B$-bimodule generated by $\partial_x$, which we call $y$, and we obtain:
\[
A= T^\ast B \cong k\langle x^{\pm 1},y \rangle  \,, \quad \ldb x,y\rdb = 1\otimes 1\,,
\]
It is easy to verify on the generators that the gauge element in $T^\ast B$ is still $\delta = xy-yx= [x,y]$, so that the zero locus is $A/\langle \delta \rangle = k[x^{\pm 1},y]$, and the corresponding commutative zero locus is a group-Lie algebra version of the commuting scheme:
\[
\Rep_n(A /\langle \delta \rangle) = \left\{ (X,Y) \in \GL_n \times \gl_n \, | \, [X,Y]=0 \right\} \,.
\]
The Shafarevich complex is obtained by adding one more variable $\vartheta$ whose differential is $[x,y]$. However we can observe that in this case, because $x$ is invertible, we can rewrite $[x,y] = (xyx^{-1} - y) x $ and at the level of matrices $XYX^{-1} = \mathrm{Ad}_X(Y)$ is the adjoint action of $\GL_n$ on $\gl_n$. Thus the Koszul complex, which a priori is the following homotopy pull-back diagram:
\[
\begin{tikzcd}
 \Spec(\C_n(A)) \arrow{r}{} \arrow{d}{}
    & \GL_n \times \gl_n \arrow{d}{}  \\
\pt  \arrow{r}{0} 
& \gl_n
\end{tikzcd}
\qquad (d \Theta = [X,Y] )
\] 
can be rewritten in a more intrinsic way as the following homotopy pull-back diagram, generalisable to other Lie algebras:
\[
\begin{tikzcd}
\gl_n \times^h_{(\gl_n\times \gl_n)} (\GL_n \times \gl_n) \arrow{r}{} \arrow{d}{}
    & \GL_n \times \gl_n \arrow{d}{{(\mathrm{Ad}}, 1)}  \\
\gl_n  \arrow{r}{\mathrm{diag}} 
& \gl_n\times \gl_n
\end{tikzcd}
\quad (d \Theta = \Ad_X(Y)- Y)
\]
We could write an analogous conjecture to the Lie algebra-Lie algebra case by saying that:
\begin{Conj} The following diagonal restriction map is a quasi-isomorphism:
\begin{equation}
\label{eq:conj1}
k [x_{ij}, y_{ij}, \vartheta_{ij} , \det(X)^{-1} ]^{\GL_n}  \we k[x_i^{\pm}, y_i,\vartheta_i]^{S_n} \,.
\end{equation}
\end{Conj}
 
\item Next we consider the following example $A=k\langle x^{\pm1},y^{\pm1}\rangle$, which is not a cotangent bundle. However we can still define a double Poisson structure by setting $ \ldb x,y\rdb = 1\otimes 1$. We can easily verify on the generators that the element $\delta = [x,y]$ defines a Hamiltonian action, that is:
\[
\ldb \delta , a \rdb= a \otimes 1 - 1 \otimes a, \quad \forall a \in A\,.
\]
The zero locus is $A/\langle \delta \rangle= k[x^{\pm 1},y^{\pm1}]$, and the corresponding commutative zero locus is a group-group version of the commuting scheme:
\[
\Rep_n(A /\langle \delta \rangle) = \left\{ (X,Y) \in \GL_n^{\times 2} \, | \, [X,Y]=0 \right\} \,.
\]
As in the previous example, using that now both matrices are invertible, we can rewrite the relation as $xyx^{-1}y^{-1} =1$, so that we can identify the Hamiltonian reduction with the $\GL_n$-character variety of the Riemann surface with genus $g=1$:
\begin{equation}
\label{eq:charRie}
\Rep_n(A/\langle \delta\rangle)\sslash \GL_n \cong \Hom_{\Grp}(\pi_1( \Sigma_1) , \GL_n)\sslash \GL_n\,.
\end{equation}
\begin{Remark}[The Poisson structure on the character variety is \emph{not} the standard one]
\label{eq:charRie}
We need to remark that the above identification~\eqref{eq:charRie} holds only at the level of affine schemes while the induced Poisson structure that we have on $\Rep_n(A\slash\langle \delta \rangle)\sslash \GL_n$ does not coincide with the one on the character variety of $\Sigma_1$. In fact, the latter is obtained by quasi-Hamiltonian reduction of the quasi-Hamiltonian space $\GL_n \acts\GL_n^{\times 2}$, equipped with the canonical $2$-form $\omega$ on the double of a Lie group (\cite{AMM}), which is different from the one induced by $\GL_n^{\times 2} \subset \gl_n^{\times 2} \cong \gl_n \times \gl_n^\ast \cong T^\ast \gl_n$. If we would like to obtain the standard Poisson structure on the character variety we would need a noncommutative analogue of the quasi-Hamiltonian formalism and Lie-group valued moment maps of \cite{AMM}, which at least in the case of the standard action of $\GL_n$ by conjugation was developed in \cite{VdB1}.
\end{Remark}
The corresponding conjecture relating the corresponding Koszul complex with its diagonal part would be (and it already appears in \cite[Conjecture 1]{BRY}):
\begin{Conj} The following diagonal restriction map is a quasi-isomorphism:
\begin{equation}
\label{eq:conj1}
k [x_{ij}, y_{ij}, \vartheta_{ij} , \det(X)^{-1},\det(Y)^{-1} ]^{\GL_n}  \we k[x_i^{\pm}, y_i^{\pm 1},\vartheta_i]^{S_n} \,.
\end{equation}
\end{Conj}
\end{enumerate}

\subsection{Decomposition of the homology of the commuting scheme}
\label{5.4}

Let $A=k\langle x,y \rangle $. Consider the following maps on the BRST complex $\B_n(A)=k[x_{ij},y_{ij},\vartheta_{ij},\eta_{ij}]$:
\begin{equation}
\label{eq:fipsi}
\begin{aligned}
&\varphi_\bullet =\{ \tr(\eta), - \}: \B_n(A)_\bullet \to \B_n(A)_{\bullet -1}\,,\\
&\psi_\bullet = \tr(\vartheta) \cdot (-) : \B_n(A)_\bullet \to \B_n(A)_{\bullet +1}\,.
\end{aligned}
\end{equation}

\begin{Prop}
The following relations are satisfied:
\begin{equation}
\label{eq:relations}
 \varphi_{\bullet+1} \psi_\bullet +\psi_{\bullet-1}\varphi_\bullet = n 1\,, \qquad \varphi_{\bullet-1}\varphi_\bullet = 0\,,\qquad \psi_{\bullet+1} \psi_\bullet= 0\, .
\end{equation}
Moreover the maps $\varphi_\bullet, \psi_\bullet$ preserve boundaries and cycles of the BRST differential, so they induce maps on the homology $\HH_\bullet(\B_n(A))$, satisfying the same relations.
\end{Prop}
\begin{proof}
For any element $\alpha \in \B_n(A)$ we have:
\[
\begin{aligned}
&\{\tr(\eta), \tr(\vartheta) \alpha \} = \{\tr (\eta) , \tr(\vartheta)\} \alpha -\tr(\vartheta) \{ \tr(\eta),\alpha\} = \tr\{\eta,\vartheta\} \alpha - \tr(\vartheta)\{ \tr(\eta), \alpha\} =\\
& \tr 1 \alpha - \tr(\vartheta) \{ \tr(\eta), \alpha\} = n \alpha - \tr(\vartheta) \{\tr(\eta), \alpha\}  \,,
\end{aligned}
\]
which proves the first relation. The second follows from the Jacobi identity and the fact that $\{ \tr(\eta),\tr(\eta)\} = \tr ( \{\eta,\eta\}) = \tr(0) = 0$. The third one is because $\tr(\vartheta)$ is an odd element, therefore $\tr(\vartheta)^2=0$. They preserve boundaries and cycles essentially because the differentials $d\vartheta= [x,y] - [\eta,\vartheta]$ and $d\eta=-\frac12 [\eta,\eta]$ are commutators, therefore $d\tr(\vartheta) = d \tr(\eta) = 0$.
\end{proof}

Moreover in the decomposition $\HH_\bullet(\B_n(A)) \cong \HH_\bullet(\C_n(A))^{\GL_n} \otimes \HH^{-\bullet} (\gl_n,k)$ they preserve the submodule $\HH_\bullet{(\C_n(A))}^{\GL_n} \subset \HH_\bullet(\B_n(A))$, and we denote by the same symbol the induced maps on the Koszul homology:
\begin{equation}
\label{eq:fipsi}
\begin{aligned}
&\varphi_\bullet : \HH_\bullet (\C_n(A))^{\GL_n}  \to \HH_{\bullet-1} (\C_n(A))^{\GL_n}\,,\\
&\psi_\bullet :  \HH_\bullet (\C_n(A))^{\GL_n}  \to \HH_{\bullet+1} (\C_n(A))^{\GL_n}\,,
\end{aligned}
\end{equation}
which satisfy the same relations~\eqref{eq:relations}. From which it follows that:
\begin{Cor}
The Koszul homology decomposes as two copies of the same graded module which are obtained one from the other by a degree $1$ shift, as follows:
\begin{equation}
\label{eq:decomp}
\HH_\bullet( \C_n(A))^{\GL_n} = \underbrace{\ker(\varphi_\bullet)}_{\bullet = 0,1,\dots ,n-1} \oplus \underbrace{ \im(\psi_{\bullet-1})}_{\bullet = 1,2,\dots, n}\, ,
\end{equation}
and the isomorphism is provided by ${\psi_\bullet}_{|_{\ker(\varphi_\bullet)}}: \ker(\varphi_\bullet )\overset{\simeq}{ \to}\im(\psi_\bullet)$.
\end{Cor}
\begin{proof} This follows precisely from the relations that the induced maps~\eqref{eq:fipsi} satisfy (which are the same as~\eqref{eq:relations}). In fact from the first and the second equation we have that any homology class $\alpha \in \HH_\bullet(\C_n(A))^{\GL_n}$ can be written as a sum of two elements:
\[
\alpha = \frac{1}{n}\underbrace{ ( \varphi_{\bullet+1} \psi_\bullet)}_{\in \ker(\varphi_\bullet)} + \frac{1}{n} \underbrace{( \psi_{\bullet-1} \varphi_\bullet) }_{\in \im(\psi_{\bullet-1})}\, .
\]
The intersection of the two submodules is zero because if $\alpha =\psi_{\bullet-1} \beta$ and $\varphi_\bullet \alpha =0$, by the third property:
\[
n \alpha =  \varphi_{\bullet+1} \psi_\bullet \underbrace{\alpha }_{\psi_{\bullet-1} \beta} +\psi_{\bullet-1}\varphi_\bullet \alpha = 0\, .
\]
Finally ${\psi_\bullet}_{|_{\ker(\varphi_\bullet)}}$ is injective because it is the left inverse of ($n$ times) the identity: $\varphi_{\bullet+1} {\psi_\bullet}_{|_{\ker(\varphi_\bullet)}} = n1$, and it is surjective on $\im(\psi_\bullet)$ because $\psi_\bullet$ is zero on the complement of $\ker(\varphi_\bullet)$.
\end{proof}

\begin{Remark}
The decomposition~\eqref{eq:decomp} can be explained also without using the trace maps in the following way. The homology of the Koszul complex $\C_n(A)$ is essentially only due to the diagonal elements ($[x,y]_{ii}$), because the other ones form a regular sequence (\cite{Kn}). Moreover of these diagonal elements one is superfluous, because their sum is zero:
\[
\tr( [x,y] ) = \sum_i [x,y]_{ii} = 0\,.
\]
It follows that the homology decomposes as the tensor product of the reduced homology (the one obtained by removing the superfluous element $\tr(\vartheta)$) and the antisymmetric algebra on the $1$-dimensional vector space generated by $\tr(\vartheta)$:
\[
\HH_\bullet(\C_n(A)) = \widetilde{\HH}_\bullet \oplus \widetilde{\HH}_{\bullet-1}\cdot \tr(\vartheta)\, .
\]
This decomposition is the same as~\eqref{eq:decomp} so, incidentally, we find another interpretation of the reduced Koszul homology of the commuting scheme as $\widetilde{\HH}_\bullet \cong \ker( \varphi_\bullet)$: the classes whose Poisson bracket with $\tr(\eta)$ vanishes.
\end{Remark}


\section{Noncommutative group actions and Poisson-group schemes}
\label{6}

In this final, short section we formalise the notions of noncommutative analogues of group schemes, group actions and Poisson-group schemes. The purpose is two-fold:
\begin{itemize}
\item On the one hand these seems to be rather natural definitions that, at least to our knowledge, did not appear yet in the literature and can help understand from a more intrinsic, coordinate-free way, many results in the paper.
\item On the other hand these notions could be used to define generalisations of the above story for noncommutative Hamiltonian actions different from the standard one (that induces the ordinary conjugation action of $\GL_n$).
\end{itemize}

\subsection{Noncommutative group schemes and actions} 
\label{6.1}

First we recall a couple of general notions about categorical groups (and cogroups), associated functors, and categorical group actions.

\textbf{Groups:} To a cartesian monoidal category $\CC$ (category with binary products and terminal object $1_\CC \in \CC$), we can associate the category $\Grp(\CC)$ of group objects in $\CC$: quadruples $(G,m,\iota,e)$ with $G\in \CC$ an object, a `multiplication morphism' $m: G\times G \to G$, an `inverse morphism'  $\iota: G \to G$, and a `unit morphism' $e: 1_\CC \to G$ satisfying the group axioms -- and morphisms the morphisms of underlying objects preserving multiplication. A (internal) group action $G\overset{\alpha}{\acts} X$ of a group $G\in \Grp(\CC)$ on $X \in \CC$ is a morphism $\alpha: G\times X \to X$ satisfying the usual two conditions of group actions. We can define a category of actions on $\CC$, which we denote by (a perhaps unconventional notation) $\Act(\CC)$: objects are group actions $G\acts X$, and morphisms from $G\overset{\alpha}{\acts} X$ to $H \overset{\beta}{\acts} Y$ are couples $(\varphi,f)$ consisting of a morphism of groups $\varphi: G \to H$ and a morphism $f:X \to Y$ such that $\beta \circ (\varphi \times f) = f \circ \alpha$. We can also fix the group $G$ (and $\varphi=\id_G$) and consider only the category of $G$-equivariant objects, which we denote by $\Act_G(\CC)$. To a cartesian functor $F: \CC \to \DD$ between cartesian monoidal categories we have induced functors between all the categories introduced above: 
\begin{enumerate}
\item $\Grp(F): \Grp(\CC) \to \Grp(\DD)$
\item $\Act(F): \Act(\CC) \to \Act(\DD)$
\item $\Act_G(F): \Act_G(\CC) \to \Act_{F(G)}(\DD)$
\end{enumerate}

\textbf{Cogroups:}  If $\AA$ has binary coproducts and initial object $\emptyset_\AA \in \AA$ (that is, $\AA^\op$ is cartesian monoidal), then we can define the category $\CoGrp(\AA)$ of cogroup objects in $\AA$: quadruples $(A,\Delta,S,\epsilon)$ with $A \in \AA$ an object, a `comultiplication' $\Delta: A \to A \coprod A$, a `coinverse' $S: A \to A$, and a `counit' $\epsilon: A \to \emptyset_\AA$ satisfying the cogroup axioms -- and morphisms the comultiplication preserving ones. Obviously $\CoGrp(\AA) \cong (\Grp(\AA^\op))^\op$. Dualising the previous constructions we can define the category of cogroup coactions $\CoAct(\AA) \left(\cong (\Act(\AA^\op))^\op \right)$ and for some fixed cogroup $A \in \CoGrp(\AA)$, the category of $A$-coequivariant objects $\CoAct_A(\AA) \left(\cong (\Act_{A^\op}(\AA^\op))^\op \right)$. To any cocartesian functor $F:\AA \to \BB$ we have induced three functors:
\begin{enumerate}
\item $\CoGrp(F)\left( \cong \Grp(F^\op)^\op\right): \CoGrp(\AA) \to \CoGrp(\BB)$
\item $\CoAct(F)\left( \cong \Act(F^\op)^\op\right): \CoAct(\AA) \to \Act(\BB)$
\item $\CoAct_A(F)\left( \cong \Act_{A^\op}(F^\op)^\op\right): \CoAct_A(\AA) \to \CoAct_{F(A)}(\BB)$
\end{enumerate}

\begin{Ex} 
\label{ex:affgroup}
Let $\CC=\Aff_k$ the category of affine $k$-schemes. The category $\Grp(\Aff_k)$ is the category of affine group schemes over $k$ (we denote it simply by $\Grp_k$), $\Act(\Aff_k)$ is the category of affine group scheme actions (we denote it by $\Act_k$), and $\Act_G(\Aff_k)$ the category of $G$-equivariant affine schemes over $k$ (we denote it by $G-\Aff_k$).
\end{Ex}

\begin{Ex}
\label{ex:op}
(= \textbf{Example}~\ref{ex:affgroup}\textsuperscript{op}). $\AA=\cAlg_k = \Aff_k^\op$. The category $\CoGrp(\cAlg_k)\cong \cHopf_k$ is the category of commutative Hopf algebras over $k$ (notice, in fact, that the antipode map of a Hopf algebra is in general an antihomomorphism of algebras, but when they are commutative, it is a homomorphism, and it corresponds to the coinverse of the cogroup structure). The category $\CoAct(\cAlg_k)$ is the category of coactions of commutative Hopf algebras on commutative algebras over $k$, and $\CoAct_A(\cAlg_k)$ the category of $A$-coequivariant commutative algebras over $k$.
\end{Ex}

\begin{Defn} [Noncommutative version of Example~\ref{ex:affgroup}] 
\label{defn:ngroup}
The category of \emph{noncommutative affine group schemes} is the category of group objects in the cartesian monoidal category of noncommutative affine schemes $\NAff_k:=Alg_k^\op$. We denote it by $\NGrp_k := \Grp(\NAff_k)$. The category of \emph{noncommutative affine group scheme actions} is the category of actions on noncommutative affines: $\NAct_k:= \Act(\NAff_k)$. For a fixed noncommutative affine group scheme $Q$ we denote the category of $Q$-equivariant objects by $Q-\NAff_k:= \Act_Q(\NAff_k)$.
\end{Defn}

\begin{Prop}
\label{prop:repgroup}
The (opposite) representation functor $\Rep_n= (-)_n^\op : \NAff_k \to \Aff_k$ is a cartesian functor, therefore it induces the following three functors, which by an abuse of notation we denote by the same symbol:
\begin{center}
\begin{enumerate}
\item $\Rep_n : \NGrp_k  \to \Grp_k $ , which enriches the usual one $\Rep_n: \NAff_k \to \Aff_k$ in the sense that there is a commutative diagram linking these two functors under the natural forgetful functors $\NGrp_k \to \NAff_k$, $\Grp_k \to \Aff_k$.
\item $\Rep_n : \NAct_k  \to \Act_k$ , 
\item $\Rep_n :  Q-\NAff_k \to \Rep_n(Q)-\Aff_k$ ,
\end{enumerate}
\end{center}
therefore justifying Definition~\ref{defn:ngroup}, according to the Kontsevich-Rosenberg principle.
\end{Prop}

\begin{Remark}
For an algebra $A \in \Alg_k$ we denote by $\Sp(A) \in \NAff_k$ the corresponding object in the opposite category (and analogously for $A \in \CoGrp(\Alg_k)$ we denote by $\Sp(A) \in \NGrp_k$ the corresponding noncommutative affine group scheme). If we do so the notation we used for the functor $\Rep_n$ in the previous Proposition is in slight contradiction with the previous notation that we used in the introduction of the paper, in which we evaluated $\Rep_n(A)$ on algebras $A$. Now instead, to be precise, we should (and will) say $\Rep_n(\Sp(A))$. 
\end{Remark}

\begin{Ex} [Noncommutative additive group $\NGa$]
This is $\NGa:= \Sp(k[x])$, where the cogroup structure on $k[x]$ is:
\[
\begin{aligned}
\Delta: &k[x] \to k \langle x_1, x_2 \rangle\\
&x \longmapsto x_1 +x_2
\end{aligned}
\qquad
\begin{aligned}
S: &k[x] \to k[x]\\
&x\longmapsto -x
\end{aligned}
\qquad
\begin{aligned}
\epsilon: &k[x] \to k \\
&x \longmapsto 0
\end{aligned}
\]
The corresponding (abelian) affine group scheme is $\Rep_n(\NGa)\cong \gl_n(k)$ (with respect to the sum).
\end{Ex}

\begin{Ex} [Noncommutative multiplicative group $\NGm$]
This is $\NGm:= \Sp(k[g^{\pm 1}])$, where the cogroup structure on $k[g^{\pm 1}]$ is:
\[
\begin{aligned}
\Delta: &k[g^{\pm 1}] \to k \langle g_1^{\pm 1}, g_2^{\pm 1} \rangle\\
&g \longmapsto g_1 g_2
\end{aligned}
\qquad
\begin{aligned}
S: &k[g^{\pm 1}] \to k[g^{\pm 1}]\\
&g \longmapsto g^{-1}
\end{aligned}
\qquad
\begin{aligned}
\epsilon: &k[g^{\pm 1}] \to k \\
&g \longmapsto 1
\end{aligned}
\]
It is interesting to observe that every noncommutative affine scheme $X =\Sp(A)$ has a somewhat natural $\NGm$-action (the precise statement is that this gives a functor $\gamma: \NAff_k \to \NGm-\NAff_k$), described dually as the coaction:
\[
\begin{aligned}
\alpha: &A \to A\ast_k k[g] = A\langle g \rangle\\
&a \longmapsto g a g^{-1}
\end{aligned}
\]
The corresponding action of the affine group scheme $\Rep_n(\NGm) \cong \GL_n(k)$ on $\Rep_n(X)$ is the standard conjugation action on representations. The composition of the functor associating to each $X$ the natural $\NGm$-action and the representation functor of Proposition~\ref{prop:repgroup}(3) is a factorisation of the representation functor with target category the $\GL_n(k)$-equivariant affine schemes:
\[
\begin{tikzcd}
	\NAff_k \arrow[rr, bend left=25, "\Rep_n"] \arrow[r, "\gamma"] &\NGm-\NAff_k \arrow [r, "\Rep_n"] &\GL_n(k)-\Aff_k
\end{tikzcd}
\]
\end{Ex}

\begin{Remark} As anticipated in the beginning of this section, the explanation of the ordinary $\GL_n$-action in these new terms tells us what we should do in case we would like to consider different actions of $\GL_n$: we consider simply $\Rep_n : \NGm -\NAff_k \to \GL_n(k) -\Aff_k$, without precomposing it with the functor $\gamma$. This means that we start from a space $\Sp(A)$ with a $\NGm$-action that is not necessarily the usual one, and in this way obtain all possible actions of $\GL_n\acts \Rep_n(A)$ that arise from noncommutative geometry. 
\end{Remark}


\subsection{Noncommutative Poisson-group schemes}
\label{6.2}

First we give the definition of affine Poisson-group schemes, which is the obvious algebro-geometric notion analogous to the differential-geometric one of Poisson-Lie groups (\cite{CP}). First we recall that the category $\CPA_k$ of commutative Poisson algebras over $k$ has binary coproducts (the underlying algebra is the tensor product) and initial object the algebra $k$ with trivial Poisson structure. Dually the category of affine Poisson schemes $\PAff_k := \CPA_k^\op$ is a cartesian monoidal category with final object $\ast=\Spec(k)$ the point. Recall that we have two forgetful functors $\PAff_k \to \Aff_k$, $\Grp_k \to \Aff_k$: one forgets the Poisson structure and the other the group structure.

\begin{Defn} The category $\PGAff_k$ of \emph{(affine) Poisson-group schemes} (over $k$) is the full subcategory of the pullback\footnote{We mean the strict pullback in the $1$-category of categories $\Cat$. In other words an object in $\PAff_k \times_{\Aff_k} \Grp_k$ is a pair of a Poisson scheme and a group scheme which are \emph{equal} under the forgetful functors, which is just a fancy way of saying that we have an affine scheme equipped with a group structure as well as a Poisson structure.} category $\PAff_k \times_{\Aff_k} \Grp_k$ consisting of those objects whose multiplication map is a Poisson map.
\end{Defn}

\begin{Remark} We remark that this is \emph{not} the same as the category of group objects in Poisson schemes $\Grp(\PAff_k)$. In fact objects of this category are Poisson schemes equipped with a group structure for which multiplication, inverse, and unit are \emph{all} Poisson maps, while $\PGAff_k$ only requires the multiplication to be a Poisson map. Moreover, one can show that a Poisson-group schemes has inverse map being a antiPoisson-homomorphism, therefore every group object in Poisson schemes has trivial (zero) Poisson structure, in other words $\Grp(\PAff_k)\cong \Grp_k$ is trivially just the category of group schemes. 
\end{Remark}

Now we have everything we need to define the noncommutative analogues of such structures. First we recall that the category $\PPA_k$ of double Poisson algebras over $k$ has binary coproducts (Proposition~\ref{prop:cop}) and initial object the algebra $k$ with trivial double Poisson structure. Dually, the category of \emph{noncommutative affine Poisson schemes} $\NPAff_k := \PPA_k^\op$ is a cartesian monoidal category with final object $\ast = \Sp(k)$, the `point'. We have two forgetful functors $\NPAff_k \to \NAff_k$, $\NGrp_k \to \NAff_k$: one forgets the (double) Poisson structure and the other the group structure.

\begin{Defn} 
\label{defn:NPGA}
The category $\NPGAff_k$ of \emph{noncommutative (affine) Poisson-group schemes} (over $k$) is the full subcategory of the pullback category $\NPAff_k \times_{\Aff_k} \NGrp_k$ consisting of those objects whose multiplication map is a Poisson map.
\end{Defn}

\begin{Remark} [Dually, algebraic side]
A noncommutative (affine) Poisson-group scheme $X=\Sp(A)\in \NPGAff_k$ is, dually, an algebra $A$ equipped with a double Poisson structure and a cogroup structure for which the comultiplication map $\Delta: A \to A \ast_k A$ is a double Poisson map.
\end{Remark}

\begin{Remark}
\label{rem:Sum}
We have the following ingredients:
\begin{itemize}
\item[(i)] $\Rep_n: \NAff_k \to \Aff_k$, the ordinary representation functor.
\item[(ii)] $\Rep_n : \NPAff_k \to \PAff_k$, enriching (i) under the natural forgetful functors~\eqref{eq:enrich}.
\item[(iii)] $\Rep_n: \NGrp_k \to \Grp_k$, enriching (i) under the natural forgetful functors (Proposition~\ref{prop:repgroup}(1)).
\end{itemize}
\end{Remark}

It follows that we have an induced pullback functor:
\begin{equation}
\label{eq:SuperRep}
\Rep_n: \NPAff_k \times_{\Aff_k}  \NGrp_k \to \PAff_k \times_{\Aff_k}  \Grp_k
\end{equation}

\begin{Thm}
The functor~\eqref{eq:SuperRep} restricts to the full subcategories of noncommutative Poisson-group schemes and Poisson-group schemes, respectively. In other words we have a factorisation:
\begin{equation}
\label{eq:SupRep}
\begin{tikzcd}
	\NPGAff_k \arrow[d, hook] \arrow[r, "\Rep_n"]& \PGAff_k \arrow[d, hook] \\
	\NPAff_k \times_{\Aff_k}  \NGrp_k \arrow[r, "\Rep_n"] & \PAff_k \times_{\Aff_k}  \Grp_k
\end{tikzcd}
\end{equation}
justifying, once again, Definition~\eqref{defn:NPGA}.
\end{Thm}

\begin{proof} 
Dually, we need to show that if $A$ is a double Poisson algebra with a cogroup structure such that the comultiplication map $\Delta : A \to A \ast_k A $ is a morphism of double Poisson algebras, then the induced comultiplication $\Delta_n: A_n \to (A\ast_k A)_n\cong A_n \otimes_k A_n$ is a morphism of Poisson algebras. This is granted once we show that the equality $(A\ast_k A )_n \cong A_n \otimes_k A_n$ holds not only at the level of algebras, but also at the level of double Poisson algebras, or in other words that the functor $\Rep_n : \NPAff_k \to \PAff_k$ (Remark~\ref{rem:Sum}(ii)) is a cartesian functor. This follows straightforwardly from the identification $(A\ast_k B)_n \cong A_n \otimes_k B_n$ and the trivial verification of the equality of the two Poisson structures on the generators.
\end{proof}


\appendix 
\section{Derived coproducts}
\label{app:A}

Let $\CC$ be a model category and $S \downarrow \CC$ the under category with respect to a fixed object $S$. The natural forgetful functor is right adjoint to the coproduct by $S$:
\[
\adjunct{\CC}{S \downarrow \CC}{S\amalg-}{U}\,.
\]
The model structure on $S \downarrow \CC$ is the one with cofibrations, weak equivalences, and fibrations the preimage of the corresponding classes under the forgetful functor (therefore making the pair $(S \amalg-, U)$ a Quillen pair). The initial object in $S \downarrow \CC$ is $S$ with identity map as structure map $S\to S$, while the final object is the final object in $\CC$, with structure map the unique map $S \to \ast$. Push-outs and pull-backs in $S \downarrow \CC$ are computed as in $\CC$ (with structure maps coming from the additional structure maps in the push-out/pull-back data). In particular the coproduct in $S \downarrow \CC$ is the push-out in $\CC$ of the diagram $ \bullet \leftarrow S \rightarrow \bullet$, and we denote this coproduct by the standard symbol $A \amalg_S B$. The coproduct in a model category (in this case $S \downarrow \CC$) is left adjoint to the diagonal functor:
\begin{equation}
\label{eq:copro}
\adjunct{(S\downarrow \CC)^{\times 2}}{S\downarrow \CC}{-\amalg_S -}{\Delta}\,,
\end{equation}
which obviously preserves all classes of maps (cofibrations, weak equivalences, fibrations). As a consequence the above pair is a Quillen pair and the coproduct has a total left derived functor, which we denote by $\L(-\amalg_S - ) = -\amalg_S^\L - $, and is a priori computed by picking cofibrant replacements for both variables. However, in the case of a left proper model category (weak equivalences are preserved by pushout along cofibrations) ordinary pushouts along cofibrations compute homotopy pushouts (\cite[Proposition A.2.4.4.(ii)]{Lu}), therefore the derived coproduct is computed by picking a cofibrant replacement of only one of the two variables:
\begin{equation}
\label{eq:1var}
A \amalg_S^\L B \cong A \amalg_S^\L QB\,,
\end{equation}
where $QB$ is a cofibrant replacement of $B$ in the under category $S \downarrow \CC$, or equivalently, a diagram $S \cofi QB \acfib B$ in $\CC$. Finally we remark that the categories we are interested in, such as dg algebras or commutative dg algebras over a field $\CC = \DGA_k , \CDGA_k$ (with coproduct being, respectively, the free product and the tensor product) are left proper model categories (see \cite[Remark 2.15]{BCL}).

\bibliographystyle{amsplain}
\bibliography{Poisson}

\end{document}